\documentclass[journal]{IEEEtran}
\IEEEoverridecommandlockouts
\usepackage{cite}
\usepackage{amsmath,amssymb,amsfonts}
\usepackage{amsthm}
\usepackage{breqn}
\usepackage{algorithm}
\usepackage{algorithmic}
\usepackage{colortbl}
\definecolor{Gray}{gray}{0.9}
\usepackage{caption}
\usepackage{graphicx}
\usepackage{subcaption}
\usepackage{multicol}
\usepackage{multirow}
\usepackage{lipsum}
\usepackage{stfloats}
\usepackage{textcomp}
\usepackage{xcolor}
\usepackage{amsmath}
\usepackage{amsfonts}
\usepackage{amssymb}
\usepackage{mathrsfs}
\usepackage{enumitem}
\usepackage{ulem}
\usepackage{booktabs}
\usepackage{arydshln}
\newtheorem{theorem}{Theorem}[section]

\newtheorem{proposition}{Proposition}[section]
\newtheorem{assumption}{Assumption}[section]
\newtheorem{lemma}{Lemma}[section]

\newcommand{\norm}[1]{\left\lVert#1\right\rVert}

\DeclareMathOperator*{\argmin}{arg\,min}

\usepackage[belowskip=-10pt,aboveskip=15pt]{caption}
\def\BibTeX{{\rm B\kern-.05em{\sc i\kern-.025em b}\kern-.08em
    T\kern-.1667em\lower.7ex\hbox{E}\kern-.125emX}}
\begin{document}

\setlength{\abovedisplayskip}{3pt}
\setlength{\belowdisplayskip}{3pt}

\title{A Distributed ADMM-based Deep Learning Approach for Thermal Control in Multi-Zone Buildings \textcolor{black}{under Demand Response Events}}

\author{Vincent Taboga$^1$ $^2$ $^3$, Hanane Dagdougui$^1$ $^2$ $^3$ \\ $^1$ Polytechnique Montreal, Department of Mathematics and Industrial Engineering \\
$^2$ Mila, Quebec Artificial Intelligence Institute \\ $^3$ Groupe d’etudes et de recherche en analyse des decisions (GERAD)}
\maketitle

\begin{abstract}

\textcolor{black}{The increasing electricity use and reliance on intermittent renewable energy sources challenge power grid management during peak demand, making Demand Response programs and energy conservation measures essential.} This research combines distributed optimization using ADMM with deep learning models to plan indoor temperature setpoints effectively. A two-layer hierarchical structure is used, with a central building coordinator at the upper layer and local controllers at the thermal zone layer. The coordinator must limit the building's maximum power by translating the building's total power to local power targets for each zone. Local controllers can modify the temperature setpoints to meet the local power targets. \textcolor{black}{While most algorithms are either centralized or require prior knowledge about the building's structure, our approach is distributed and fully data-driven. The proposed algorithm, called Distributed Planning Networks, is designed to be both adaptable and scalable to many types of buildings, tackling two of the main challenges in the development of such systems.} The proposed approach is tested on an 18-zone building modeled in EnergyPlus. The algorithm successfully manages Demand Response peak events. 

\end{abstract}

\def\abstractname{Note to Practitioners}
\begin{abstract}

This article addresses the problems faced by multizone buildings in maintaining thermal comfort while participating in Demand Response events. A centralized control approach is nonscalable, raises privacy concerns, and requires huge communication bandwidth. This work introduces a distributed algorithm that adapts the temperature setpoints of multi-zone buildings to enhance Demand Response participation and energy conservation. The proposed approach is fully data-driven and only requires historical data on weather, indoor temperature, and HVAC power consumption. The blueprint of the building and details about the HVAC system's architecture are not needed. Energy savings will be spread across all the zones depending on the user's choice of comfort temperature interval for each zone. The distributed optimization makes the approach scalable to large buildings with multiple zones. The algorithm is designed to work for buildings with various insulated zones such as apartment buildings and hotels, and may be extended to other types of buildings.
\end{abstract}



\section{Introduction}

\IEEEPARstart{B}{uildings} account for approximately $40\%$ of the energy consumption worldwide. More than half of the building's energy is used to maintain the occupant's comfort, mainly through Heating, Ventilation and Air Conditioning (HVAC) systems \cite{levesque2018}. Buildings are crucial for integrating energy efficiency and Demand Response (DR) measures to reduce electricity consumption, especially during unstable power supply and peak periods. Many paths are investigated to benefit grid management by leveraging the building's power consumption. Local renewable energy sources, energy storage systems, and electric vehicles can now be embedded in the Building Energy Management System (BEMS) to form a building-integrated microgrid \cite{Fontenot2019, Ouammi2020}. Efforts are also made to improve insulation materials and directly reduce the energy needs for heating and cooling purposes \cite{Belussi2019}. In addition, many grid operators are now proposing DR programs to encourage consumers to reduce their energy consumption following an incentive-based or dynamic pricing program. In both cases, they propose rewards to the end-users in exchange for reductions in their power or energy consumption, especially during periods of high solicitation of the grid. Thus, smart control of the building loads may result in substantial financial advantages for the consumers while providing support for an efficient grid operation \cite{Dafonseca2021, Zhang2023, Xie2023}. Heating and cooling demands can significantly impact the total building's power consumption profile. The thermal inertia of the building offers enhanced operational flexibility to reduce/shift a part of the load using pre-cooling or pre-heating strategies. The use of smart temperature controllers, accounting for occupancy and weather variations, results in a direct reduction of energy usage. However, it has a direct impact on the user's comfort that must be preserved \cite{Dafonseca2021}.
Many works have focused on the creation of advanced BEMS for monitoring and controlling a building's energy requirements, including smart temperature controllers. 
Two of the main approaches used to develop BEMS are Model Predictive Control (MPC) and Reinforcement Learning (RL). For instance, \cite{Deconinck2016, Rezaei2020, tanaskovic2017} apply MPC to manage energy consumption and comfort in commercial and residential buildings. Works such as \cite{Vazquezcanteli2019, Zhang2019, Taboga2021, Du2021, Zhuang2023} use RL to create efficient and adaptive controllers for BEMS.

Optimization problems solved using MPC or RL are often formulated as centralized problems. In centralized problems, a single entity computes the solution to the problem and makes all the decisions. Centralized problem formulations for BEMS are further described in \cite{yao2021}. Although simple conceptually, this approach is not scalable as the system size grows \cite{Yu2021}. Furthermore, the necessary communications between local actuators and the central controller may raise privacy concerns and failure risks. Thus, distributed architectures are increasingly considered. Distributed architectures are of two types: the hierarchical approach and the fully distributed approach (peer-to-peer). In both approaches, the complex global optimization problem is decomposed into subproblems.
In hierarchical structures, the problem is solved through a coordinator who takes global decisions and communicates with Local Controllers (LCs). In fully distributed architectures, LCs directly communicate with each other to coordinate their actions. The choice of architecture is problem-dependent, as each architecture has pros and cons and highly depends on the communication infrastructure \cite{Drgona2020,Vazquezcanteli2019}. 

Distributed architectures have been increasingly used along with the Alternating Direction Method of Multipliers (ADMM) to tackle energy management challenges at the scale of the building. MPC is often considered to solve the optimization problems resulting from a decomposition, yielding a class of algorithms called Distributed Model Predictive Control (DMPC). Works using DMPC for HVAC control are further discussed in \cite{yao2021}. For instance, \cite{walker2017, Hou2017} present case studies with linearized building dynamics, where DMPC provides equivalent performances as centralized MPC while requiring fewer computational resources. In \cite{Wang2022}, Wang et al. use a multiple-layer distributed architecture to operate HVAC systems. Their model accounts for complex coupling dynamics between the components of the HVAC system. The focus of their study is more on optimizing the operation of HVAC systems given comfort requirements. The control of the zones' temperature setpoints to act on the power consumption has not been considered. In addition, the authors used simple linear models for the thermal dynamics of the building. In their study, Mork et al. \cite{Mork2022} successfully implement nonlinear DMPC to control a multi-zone building. Non-linear models are seldom used in the literature but are more precise because the thermal transfer processes are non-linear \cite{Mork2022, Drgona2021}. They use a physics-based model developed in Modelica, that accounts for the thermal and hydraulic coupling between zones. Their case studies show the efficiency of the distributed approach over centralized control. One drawback of their approach is that it requires enough knowledge about the building's blueprint to build the model. Some data required for modelling the coupling between zones, such as door and shading positions, might also be hard to access. In addition, they apply ADMM to non-convex problems without analyzing the convergence properties of their algorithm.

On the theoretical side, ADMM is proven to converge when the underlying problem is convex \cite{Boyd2011}, which is not the case in BEMS applications. To rely on the convergence guarantees of convex problems, many works linearize the dynamics of their system \cite{walker2017, Hou2017, Yang2021, Rezaei2022, Ma2011, Zhang2017}. Others rely on the empirical robustness of ADMM when applied to non-convex problems \cite{Mork2022}. Few theoretical results exist on the convergence of ADMM in non-convex settings \cite{Hong2014, Hong2018AApproach}. The use of a non-convex ADMM allows to consider models more sophisticated than linear ones. In particular, models relying on neural networks proved to be extremely accurate in forecasting the energy consumption in buildings \cite{Lu2022, Drgona2021}. Recurrent State Space Models (RSSM) also proved to be extremely accurate in modeling stochastic environments and are yet to be applied to building energy forecasting. Such models, along with planning algorithms provide state-of-the-art results in many benchmark control tasks \cite{Hafner2019, Hafner2020}.
\vspace{0.2cm}

In line with the International Energy Agency's (IEA) recommendations \cite{IEA2022}, the present study focuses on the critical role of widening temperature setpoints as a way for significant energy savings and effective implementation of DR programs. While optimizing the HVAC system's operation to efficiently track a setpoint is important, it is crucial to note that for a given temperature setpoint, there exists an inherent physical constraint on the extent of energy reduction achievable solely through the optimization of HVAC system operations. Further energy reduction may only be achieved by modifying the setpoint itself. This study thus uses temperature setpoints as actuators for power consumption regulation. Within the proposed framework, users select a preferred temperature setpoint and define an associated comfort range for each zone within a building. The BEMS dynamically adjusts the setpoints within the comfort intervals to distribute energy savings across the building. 

Our approach combines the recent advancement in non-convex ADMM algorithms, with state-of-the-art control methods based on deep learning models. First, we formulate the temperature control problem as a centralized problem with a constraint on the building's total power consumption. Then, we convert this problem into a non-convex sharing problem that we decompose using the ADMM algorithm. We provide convergence proof of the algorithm for this non-convex setting, which is key to accurately account for the non-linear dynamics of buildings. From this formulation, we derive a surrogate problem that can be efficiently solved in practice. We test both State Space Models (SSM) and RSSM to model the building's zone. The models are used to plan for the temperature setpoints that best match the comfort requirements of zones while enforcing a limit on the building's power usage. The two models are compared in terms of prediction accuracy to benchmark RSSM in a building environment, and in terms of control quality to assess the performance of our distributed control framework. The main contributions of this study are summarized as follows:

\color{black}

\begin{enumerate}
    \item  We combine non-convex ADMM and fully data-driven deep-learning models to develop a distributed control algorithm for HVAC systems. A convergence analysis of the ADMM algorithm in this non-convex setting is presented. To the best of our knowledge, this is the first work combining non-linear deep-learning models and ADMM.
    \item We propose a hard constraint of the maximum HVAC power consumption to prioritize energy savings during peak periods, while the use of comfort intervals instead of fixed temperature setpoints guarantees a minimum level of comfort for the users.
    \item We derive a deterministic and a stochastic version of our algorithm. The performances of both approaches are evaluated on a multi-zone building during DR events. The results demonstrate the control system's capability to lower power consumption during peak periods. The algorithm consistently converges to a solution with each iteration. Additionally, the distributed architecture of the control system improves its scalability as the building size increases.
\end{enumerate}

In section \ref{sec:problem_formulation}, we introduce the control problem and derive a distributed algorithm to solve it. In section \ref{sec:disctributed_planning_networks}, we present the deep learning models used to model the thermal zones and the implementation of the proposed distributed control algorithm. In section \ref{sec:experiments}, we present experimental results on the prediction accuracy of the models and the control performances of the algorithm. Finally, we discuss the results and some limitations paving the way for future works in section \ref{sec:discussion_and_further_work}.

\color{black}


\section{Description of the Control Problem}
\label{sec:problem_formulation}

\subsection{Problem Formulation}
\label{subsec:problem_formulation}

\color{black}

Consider a building composed of $N$ thermal zones equipped with independent HVAC systems. In each zone, occupants may choose a temperature setpoint $T^{sp}_{t,i}$ that is optimal for their comfort, where $t$ is a time index and $i$ a zone index. These setpoints may vary in time but are known beforehand. To track their setpoint, each HVAC system consumes a given amount of power $u_{t,i}$. The HVAC power consumption is directly impacted by the temperature setpoints  \cite{Sunardi2020} and we propose to use the setpoints as actuators to control the HVAC consumption. Specifically, we aim at changing the occupants' preferred setpoints $T^{sp}_{t,i}$ by $\delta_{t,i}$, so that the HVAC system tracks the setpoint $T^{sp}_{t,i} + \delta_{t,i}$. \color{black} Let $$ \boldsymbol{\Delta} = [\delta_{1,1}, \ldots, \delta_{1,N},\ldots, \delta_{H,1}, \ldots \delta_{H,N}]^T $$ be a vector containing the setpoint changes for all the zones on a horizon of $H$ timesteps. Let $ \boldsymbol{\Delta}_t = [\delta_{t,1}, \ldots \delta_{t,N}]^T$ be the vector containing all setpoint changes at time $t$. Similarly, $ \boldsymbol{\Delta}_i$ contains all the temperature setpoint changes for zone $i$ over the horizon. This indexing convention will be used for other vectors as well. In particular, the vector of power consumption: $$ \textbf{u} = [u_{1,1}, \ldots, u_{1,N},\ldots, u_{H,1}, \ldots u_{H,N}]^T. $$ \color{black}

The purpose of the control system is to apply a constraint on the entire building's HVAC maximum power consumption. We define the optimal setpoint changes $ \boldsymbol{\Delta}$ to satisfy this constraint as solutions to the following optimization problem:

\color{black}

\begin{align}
\underset{\boldsymbol{\Delta}}{\text{minimize}} \quad & \|\boldsymbol{\Delta}\|^2_2 \label{opt:original_control} \\
\text{subject to} \quad & \boldsymbol{s}_{t+1, i} = f_i \left(\boldsymbol{s}_{t,i}, \delta_{t,i}, \boldsymbol{s}_{y, j \in \mathcal{N}_i}, \delta_{t,j \in \mathcal{N}_i} \right), \forall t,i \label{opt:original_control_c1} \\
& m_i \leq \delta_{t,i} \leq M_i, \forall t,i \label{opt:original_control_c2} \\
& A \textbf{u} \leq \textbf{P}^{max} \label{opt:original_control_c3}
\end{align}
where 

\begin{itemize}[noitemsep, topsep=0pt]
    \item $f_i$ represents the dynamics of zone $i$,
    \item $\bold{s}_{t, i}$ is the state of zone $i$. The state includes, in particular, $T^{sp}_{t, i}$, the air temperature, and the HVAC power consumption. Note that the state is partially observable as it may also include the temperature of the walls or the thermal capacity of the zone.
    \item $\mathcal{N}_i$ is the set of neighbors of zone $i$,
    \item $A \in \mathbb{R}^{H \times HN}$ is a matrix that sums the power consumption of all the zones at each timestep. For instance, for two zones and a horizon of two timesteps:  $$ A = \begin{bmatrix}
1 & 1 & 0 & 0\\
0 & 0 & 1 & 1
\end{bmatrix}. $$
    \item  $\textbf{P}^{max} = [P^{max}_1, \ldots, P^{max}_H]$ is a vector containing the maximum power for each timestep. The vector inequality in Eq. (\ref{opt:original_control_c3}) is component-wise.
    \item \textcolor{black}{$m_i$ and $M_i$ are the minimum and maximum setpoint changes defined by the occupants in zone $i$.}
\end{itemize}
\color{black} This formulation prioritizes power savings over temperature comfort as long as the temperature setpoint stays in the comfort interval $\left[T^{sp}_{t,i} - m_i; T^{sp}_{t,i} + M_i \right]$ defined by the user.

To solve Problem (\ref{opt:original_control}) efficiently, we reformulate it as a sharing problem \cite{Hong2014} and apply ADMM to solve it. 
\color{black} First, the inequality constraint (\ref{opt:original_control_c3}) is relaxed using a new parameter $\textbf{P}^{tot}$ such that $\textbf{P}^{tot} \leq \textbf{P}^{max}$. This ensures the problem is feasible and removes a coupling constraint. Second, the true dynamics $f_i$ is approximated by a function $q_i$ that does not take the neighboring states as input. This assumption implies that one can reasonably predict a zone's state without having to predict the states of neighboring zones. Note that the initial state of a zone $i$, $s_{0,i}$ may still be augmented to include information about the neighbors at $t=0$. In many buildings such as hotels, offices, or residential buildings, the temperature ranges in the zones are similar, thus limiting the heat exchanges. We will assess this assumption experimentally in section \ref{sec:experiments} to show that it is reasonable if the zones are well insulated. Such an assumption is also discussed in \cite{Mork2022, Lamoudi2011}. The optimization problem becomes:
 
\begin{align}
\underset{\boldsymbol{\Delta}}{\text{minimize}} \quad & \| \boldsymbol{\Delta}\|^2_2 + ||A\textbf{u} - \textbf{P}^{tot}||^2_2 \label{opt:const_problem} \\
\text{subject to} \quad & \boldsymbol{s}_{t+1, i} = q_i \left(\boldsymbol{s}_{t,i}, \delta_{t,i} \right), \quad \forall t,i \label{opt:const_problem_c1} \\
& m_i \leq \delta_{t,i} \leq M_i, \forall t,i
\end{align}

\textcolor{black}{Note that we do not use weights to balance the two terms of the objective. In practice, the temperature and power values should be normalized for the two terms to have the same order of magnitude.} The HVAC power consumption is now driven towards the parameters $\textbf{P}^{tot}=[P_1^{tot}, \ldots, P_H^{tot}]$. The choice of $\textbf{P}^{tot}$ is important and will be further discussed in section \ref{sec:maximum_power_constraint_conversion}. 

\color{black}
The powers $\textbf{u}$ are components of the state $\textbf{s}$ and the constraint Eq. (\ref{opt:const_problem_c1}) can be made implicit by unrolling the dynamics $q_i$ from the initial state using the setpoint schedules. To this purpose, let $\phi_i$ map the state of a zone $i$ to its power usage
$ \displaystyle    u_{t,i} = \phi_i(\bold{s}_{t, i})$, and consider the following functions for all $i$ and $t$:
\color{black}

\begin{flalign}
   & Q_{t,i}(\bold{s}_{0, i},  \boldsymbol{\Delta}_i) =  q_i(q_i(\ldots q_i(\bold{s}_{0, i}, \delta_{0, i})), \delta_{t, i}) \label{eq:state_dynamics}\\
   & g( \boldsymbol{\Delta}_i) = \norm{ \boldsymbol{\Delta}_{i}}^2_2 \\
   & \ell( \boldsymbol{\Delta}) =  \sum_{t=1}^H \left (\sum_{i=1}^N \phi_i(Q_{t,i}(\bold{s}_{0,i},  \boldsymbol{\Delta}_i) - \textbf{P}^{tot}_t \right)^2 \label{eq:l_function}
\end{flalign}

The index to the function $Q_{t,i}(\bold{s}_{0,i}, \boldsymbol{\Delta}_i)$ indicates that it represents the state at time $t$ of zone $i$ computed from $\bold{s}_{0,i}$. Only the first $t$ elements of $ \boldsymbol{\Delta}_i$ are needed to compute $\bold{s}_{t,i}$, but the entire vector is passed for notational simplicity. The power usage of a zone is thus
\begin{equation}
    \label{eq:relation_power_setpoint}
    u_{t,i} = \phi_i ( Q_{t,i}(\bold{s}_{0,i}, \boldsymbol{\Delta}_i)).
\end{equation}

Using these notations, Problem (\ref{opt:const_problem}) may be reformulated as a sharing problem

\begin{align}
\underset{\boldsymbol{\Delta}}{\text{minimize}} \quad & \sum_{i=1}^N g(\boldsymbol{\Delta}_i) + \ell(\boldsymbol{\Delta}) \label{opt:sharing_problem} \\
\text{subject to} \quad & m_i \leq \delta_{t,i} \leq M_i \quad \forall t,i.
\end{align}

Each zone has a local objective represented by $g$ and a global objective represented by $\ell$. The total power available is a common resource that must be shared across all zones to satisfy the local comfort requirements.

\color{black}

\subsection{Distributed Sharing Algorithm}
\label{subsec:distributed_sharing_algorithm}

In this section, we derive a distributed algorithm based on ADMM to solve the Problem (\ref{opt:sharing_problem}). Note that no particular assumption is made on the dynamics $q_i$ of the zones and the problem is thus non-convex.
\color{black} First, we introduce duplicated variables $\Bar {\boldsymbol{\Delta}}_{1}, \ldots, \Bar{ \boldsymbol{\Delta}}_{N}$ and form an equivalent problem with $N$ vector linear equality constraints:
\begin{align}
\underset{\boldsymbol{\Delta}}{\text{minimize}} \quad & \sum_{i=1}^N g(\boldsymbol{\Delta}_i) + \ell \left(\sum_{i=1}^N B_i \Bar{\boldsymbol{\Delta}}_i \right) \label{opt:admm_problem} \\
\text{subject to} \quad & \Bar{\boldsymbol{\Delta}}_i = \boldsymbol{\Delta}_i \quad \forall i \\
& m_i \leq \delta_{t,i} \leq M_i \quad \forall t,i
\end{align}

The matrices $B_i \in \mathbb{R}^{NH \times H}$ for $i=1, \ldots, N$ are such that $ \boldsymbol{\Delta} = \sum_{i=1}^N B_i \boldsymbol{\Delta}_i$. The augmented Lagrangian, with Lagrange multipliers $\boldsymbol{\lambda}_{1}, \ldots, \boldsymbol{\lambda}_{N} \in \mathbb{R}^H$ associated with the equality constraints, is:

\begin{align}
    \mathcal{L}_\rho(\boldsymbol{\Delta}, \Bar{\boldsymbol{\Delta}}, \boldsymbol{\lambda}_{1}, \ldots, \boldsymbol{\lambda}_{N}) = \sum_{i=1}^N g(\boldsymbol{\Delta}_i) + \ell(\Bar{\boldsymbol{\Delta}}) \nonumber \\ 
    + \sum_{i=1}^N \left( \boldsymbol{\lambda}_{i}^T(\Bar{\boldsymbol{\Delta}_{i}}-\boldsymbol{\Delta}_{i}) + \rho \|\Bar{\boldsymbol{\Delta}_{i}}-\boldsymbol{\Delta}_{i}\|^2_2 \right).
\end{align}

\begin{algorithm}[b]
\caption{:non-convex consensus ADMM} \label{alg:non_convex_admm}
    \begin{algorithmic}[1]
    \STATE $\text{Initialize } \{\boldsymbol{\Delta}_i\}_{i=1,\ldots,N}, \{\Bar{\boldsymbol{\Delta}}_i\}_{i=1,\ldots,N}, \{\boldsymbol{\lambda}_i\}_{i=1,\ldots,N}$
    \STATE $k=0$
    \WHILE{has not converge}
        \FOR{$i=1,\ldots,N$}
        \STATE $\Delta_i^{k+1} \gets \argmin_{\boldsymbol{\Delta}_i} g(\boldsymbol{\Delta}_i) + \boldsymbol{\lambda}_i^{k,T}(\Bar{\boldsymbol{\Delta}}_i^k-\boldsymbol{\Delta}_i) \newline
        \hspace*{2.5cm}  + \frac{\rho}{2} ||\Bar{\boldsymbol{\Delta}}_i^k - \boldsymbol{\Delta}_i||^2_2$
        \ENDFOR
        \STATE $\Bar{\boldsymbol{\Delta}}^{k+1}  \gets \argmin_{\Bar{\boldsymbol{\Delta}}} \ell \left( \sum_{i=1}^N B_i \Bar{\boldsymbol{\Delta}}_i \right) \newline
        \hspace*{0.5cm} + \sum_{i=1}^N ( \boldsymbol{\lambda}_i^{k, T} (\Bar{\boldsymbol{\Delta}}_i-\boldsymbol{\Delta}_i^{k+1}) +\frac{\rho}{2} ||\Bar{\boldsymbol{\Delta}}_i- \boldsymbol{\Delta}_i^{k+1}  ||^2_2$ )
        \FOR{$i=1,\ldots,N$}
        \STATE $\boldsymbol{\lambda}_i^{k+1} \gets \boldsymbol{\lambda}_i^k + \rho (\Bar{\boldsymbol{\Delta}_i}^{k+1} - \boldsymbol{\Delta}_i^{k+1})$
        \ENDFOR
    \ENDWHILE
        
    \end{algorithmic}
\end{algorithm}

The sharing problem can be solved using Algorithm \ref{alg:non_convex_admm}. We distinguish two main sub-problems in this algorithm. The first one is the Local Controller (LC) problem (line 5) and the second one is the coordinator problem (line 7).

As stated in the following theorem, under some Lipschitz regularity assumptions and the right choice of $\rho$, this algorithm is guaranteed to converge although the problem is not convex. We will assess the quality of the solutions experimentally in section \ref{sec:experiments}. \textcolor{black}{A proof of the theorem is given the Appendix \ref{subsec:proof_theorem_1}}.

\begin{theorem}
\label{theorem:1}
The Algorithm \ref{alg:non_convex_admm} converges to the set of stationary solutions of Problem (\ref{opt:admm_problem}): 

$$ \lim_{k\xrightarrow{}\infty} dist(\{\boldsymbol{\Delta}_i^k\}, \{\Bar{\boldsymbol{\Delta}}_i^k\}, \{\boldsymbol{\lambda}_i^k\}; Z^*) = 0,$$

where $Z^*$ is the set of primal-dual stationary solutions of the problem. 

\end{theorem}

Querying the zones' dynamics $q_i$ is computationally expensive. To leverage the benefits of the decomposition, the dynamics' evaluation should be placed in the for loop and computed in parallel. In Algorithm \ref{alg:non_convex_admm} the dynamics are called in the $\ell$ function, in the coordinator problem. \color{black} It means that the local controllers are sending their temperature setpoint schedules to the coordinator which has to query all the zones' dynamics to compute the associated HVAC consumption. Algorithm \ref{alg:non_convex_admm} formulation is convenient for a theoretical analysis of the problem. In practice, however, it is more efficient to compute the HVAC consumption at the LC level and exchange power consumption messages with the coordinator. We thus explicit the relationship between the setpoints and the related power consumption given by Eq. (\ref{eq:relation_power_setpoint}) as a constraint in the LC problem and derive a surrogate LC problem: \color{black}

\begin{align}
\underset{\boldsymbol{\Delta}_i}{\text{minimize}} \quad & g(\boldsymbol{\Delta}_i) + \boldsymbol{\lambda}_i^T (\Bar{\textbf{u}}_i - \textbf{u}_i) + \rho || \Bar{\textbf{u}}_i - \textbf{u}_i ||_2^2 \label{opt:local_control} \\
\text{subject to} \quad & u_{t,i} = \phi(Q_{t,i}(\boldsymbol{s}_{0, i}, \boldsymbol{\Delta}_i)) \label{opt:local_control_c1} \\
& \bar{u}_{t,i} = \phi(Q_{t,i}(\boldsymbol{s}_{0, i}, \bar{\boldsymbol{\Delta}}_i)) \label{opt:local_control_c2} \\
& m_i \leq \delta_{t,i} \leq M_i \label{opt:local_control_c3}
\end{align}

as well as a surrogate coordinator problem:

\begin{align}
    \underset{\Bar{\textbf{u}}}{\text{minimize}} \quad \Tilde{\ell}(\Bar{\textbf{u}}) + \sum_{i=1}^N \boldsymbol{\lambda}_i^T (\Bar{\textbf{u}}_i - \textbf{u}_i) + \rho || \Bar{\textbf{u}}_i - \textbf{u}_i ||_2^2 \label{opt:coordinator}
\end{align}

where $\Tilde{\ell} (\bar{\textbf{u}}) = \norm{A \bar{\textbf{u}}- P^{tot}}$. In our implementation, these surrogate problems are used in place on the LC and coordinator problem in Algorithm \ref{alg:non_convex_admm}. The resulting control architecture is summarized in Fig. \ref{fig:control_architecture}.

Note that the surrogate formulation may be derived from the same decomposition steps as the ones presented in Section \ref{subsec:problem_formulation}, but using the following equation rather than Eq. (\ref{eq:relation_power_setpoint}):

\begin{equation}
    \label{eq:relation_setpoint_power}
    \delta_{t,i} = \psi(Q_{t,i}(s_{0,i}, u_i)),
\end{equation}

Eq. (\ref{eq:relation_setpoint_power}) bounds the state and power usage to a setpoint. The problem is that $\psi$ is not an injective application. For instance, if the temperature in a zone is $20^\circ C$, temperature setpoints below $20^\circ C$ will all yield a heating power of $0W$. Such an application is impossible to learn accurately in practice. That said, depending on the type of setpoint tracking system, how the power is measured and the operation regime, one may find conditions for $\psi$ to satisfy the requirements of Theorem 1. For instance, a proportional controller that operates with a zone's temperature around the setpoint or an ON/OFF controller used with Pulse Width Modulation would give good properties to $\psi$. Instead of enumerating such conditions, that are hard to verify in practice, we provide a theoretical analysis for the Problem (\ref{opt:sharing_problem}). This formulation arises naturally from Problem (\ref{opt:original_control}) and indicates the form of the ADMM decomposition for the surrogate problem. As shown in Section \ref{sec:experiments}, the surrogate formulation yields efficient computation and shows excellent convergence properties experimentally. In addition, it provides a good intuition about the underlying mechanism: the coordinator (i.e. Problem (\ref{opt:coordinator})) is choosing the amount of power for each room and each timestep, to ensure that the power constraint is always satisfied. The LCs (i.e. Problems (\ref{opt:local_control})) must then find the temperature setpoints to track this power target while preserving the users' comfort. Even if multiple setpoints may satisfy the constraints Eq. (\ref{opt:local_control_c1}) and (\ref{opt:local_control_c2}), the lowest setpoint changes will be chosen because of the term $g(\Delta_i)$ is the objective function.

\begin{figure}[t]
   \centering
   \includegraphics[width=0.4\textwidth]{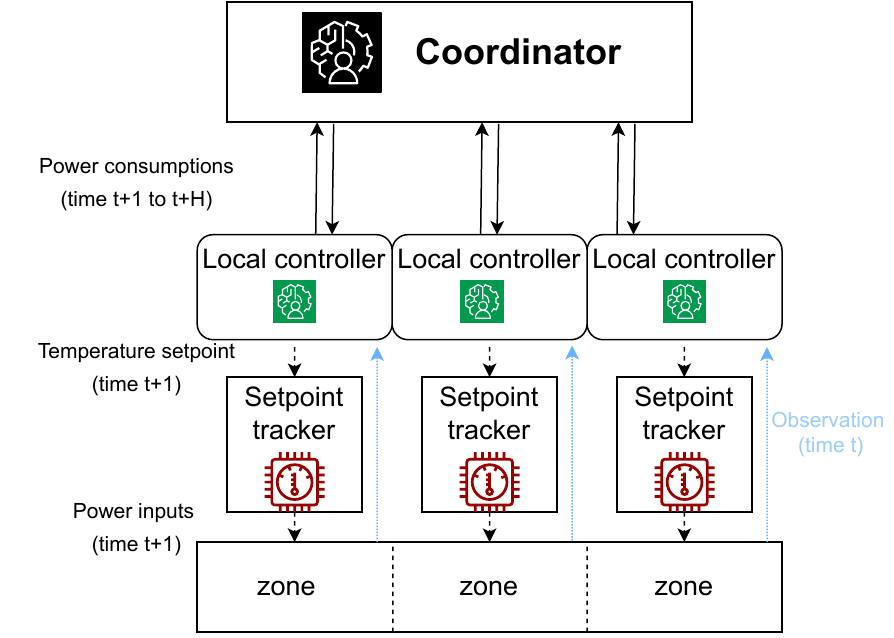}
   \caption{Two layers hierarchical control architecture. The LCs observe the state of their zone and plan for the best setpoints while communicating with the coordinator to enforce a maximum constraint on power usage. Setpoints are sent to the existing setpoint trackers that act on each zone.}
   \label{fig:control_architecture}
\end{figure}

\color{black}

\section{Distributed Planning Networks}
\label{sec:disctributed_planning_networks}

In this section, we present the methods used to solve the optimization problems of Algorithm \ref{alg:non_convex_admm}. In practice, the dynamics $q_i$ of the zones are unknown. In this study, we consider deep learning models to learn the dynamics from data. Both a deterministic and a stochastic approach are used, as described in the first part of this section. The models are used for planning to solve the LC problem (\ref{opt:local_control}), yielding the Distributed Planning Networks (DPN) algorithm. In the second part this section, we present a deterministic and a stochastic version of the DPN algorithm.

\subsection{Deep Learning Models for Thermal Zones}

\color{black}

\begin{figure*}[t]
\centering
  \includegraphics[width=0.6\textwidth, height=7.5cm]{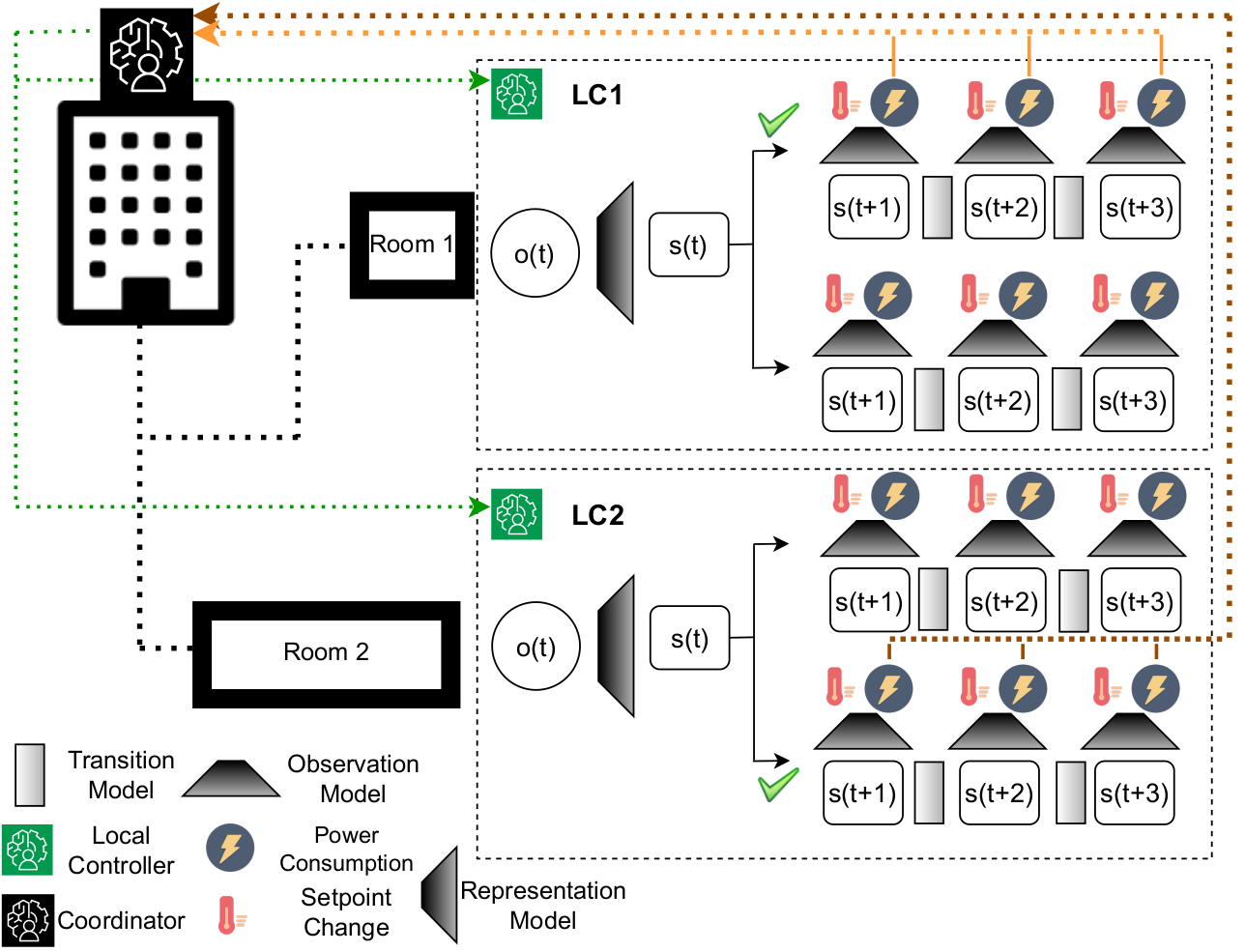}
  \caption{Distributed Planning Network - For each iteration of the ADMM algorithm, the coordinator sends power constraints to the LCs. Using the last observation, the LCs plan for the best trajectory of actions, either by shooting or random search, and send the predicted powers to the coordinator. Once the ADMM has converged, only the first action of the horizon is applied to the zones.}
  \label{fig:distributed_planning_networks}
\end{figure*}

\color{black} To plan for the optimal setpoint changes, LCs must be able to predict the evolution of the thermal zones' state. In this section, we present two types of models for the zones: State Space Models (SSM) and Recurrent State Space Models (RSSM). These models use inputs that include the observation of the zone $\textbf{o}_{t,i}$, actions $\textbf{a}_{t,i}$ and disturbances $\textbf{d}_t$ like weather and calendar information, to predict the next state in the observable space. To not overload the prediction models, the disturbances are not predicted by the models but given as inputs, as pictured in Fig. \ref{fig:prediction_scheme}. One advantage of this approach is that a single weather prediction system is required for the entire building, and the zones' model can focus on predicting only the zones' states. The zones are partially observable environments and making predictions in a latent space helps to cope with partial observability. 

\subsubsection{State Space Models}

The SSM has an encoder-decoder structure and the transition function is a recurrent cell that takes as input the actions and disturbances. The latent state $s_t$ is carried by the hidden state of the recurrent cell:

\begin{align}
    \bold{s}_{t+1} &= f_{\boldsymbol{\theta}}(\bold{s}_t, \bold{a}_t, \bold{d}_t) \label{eq:z_dynamics} \\
    \bold{o}_{t+1} &= \text{dec}_{\boldsymbol{\theta}}(\bold{s}_{t+1}) \label{eq:s_dynamics} \\
    \bold{s}_0 &= \text{enc}_{\boldsymbol{\theta}}(\bold{o}_{-1}, \ldots \bold{o}_{-n}) \label{eq:initial_condition}
\end{align}

where $\text{enc}$ is an recurrent neural network encoding lags of observations, $\text{dec}$ is a neural network decoding the latent state and $f$ is a recurrent cell used as the transition function. Prior works such as \cite{Drgona2021} have used similar models for representing buildings' dynamics. The model is trained to minimize the mean squared error between the predicted and actual powers on the entire prediction horizon.

\color{black}

\subsubsection{Reccurent State Space Model}

The RSSM is a stochastic model that allows forecasting of the power usage of each zone under uncertainties brought by many unobservable disturbances. The use of RSSM for planning is further described in \cite{Hafner2019}.  In this section, we focus on some key elements for completeness.

The RSSM is a black-box model, learned from previously collected data. It is composed of three main parts:

\begin{enumerate}
    \item A representation model (encoder), which maps the observations $\bold{o}_t$ to the latent state $\bold{s}_t$ of the system. 
    \item A transition model, that predicts the evolution of the latent space. Note that this transition model is composed of two parts: (1) a deterministic state model which predicts the evolution of the deterministic part of the state $\boldsymbol{h}_t = f(\bold{h}_{t-1}, \bold{s}_{t-1}, \delta_{t}, \bold{d}_t)$ and (2) a stochastic state model which samples the next latent state $\bold{s}_t \sim p(\bold{s}_t|\bold{h}_t)$. The latter part is important in our case as it accounts for the perturbations due to the occupancy and neighboring zones.
    \item An observation model (decoder), that outputs the observation corresponding to a given latent state.
\end{enumerate}

All these components are represented by neural networks. We use the same set of parameters $\theta$ to learn the observation model, the transition model, and the representation model. The objective is to maximize the variational bound as proposed in \cite{Hafner2019}. The transition model predicts one timestep ahead in the latent space. Multiple-step prediction is achieved by applying the model recursively.

\begin{figure}
   \centering
   \includegraphics[width=0.4\textwidth]{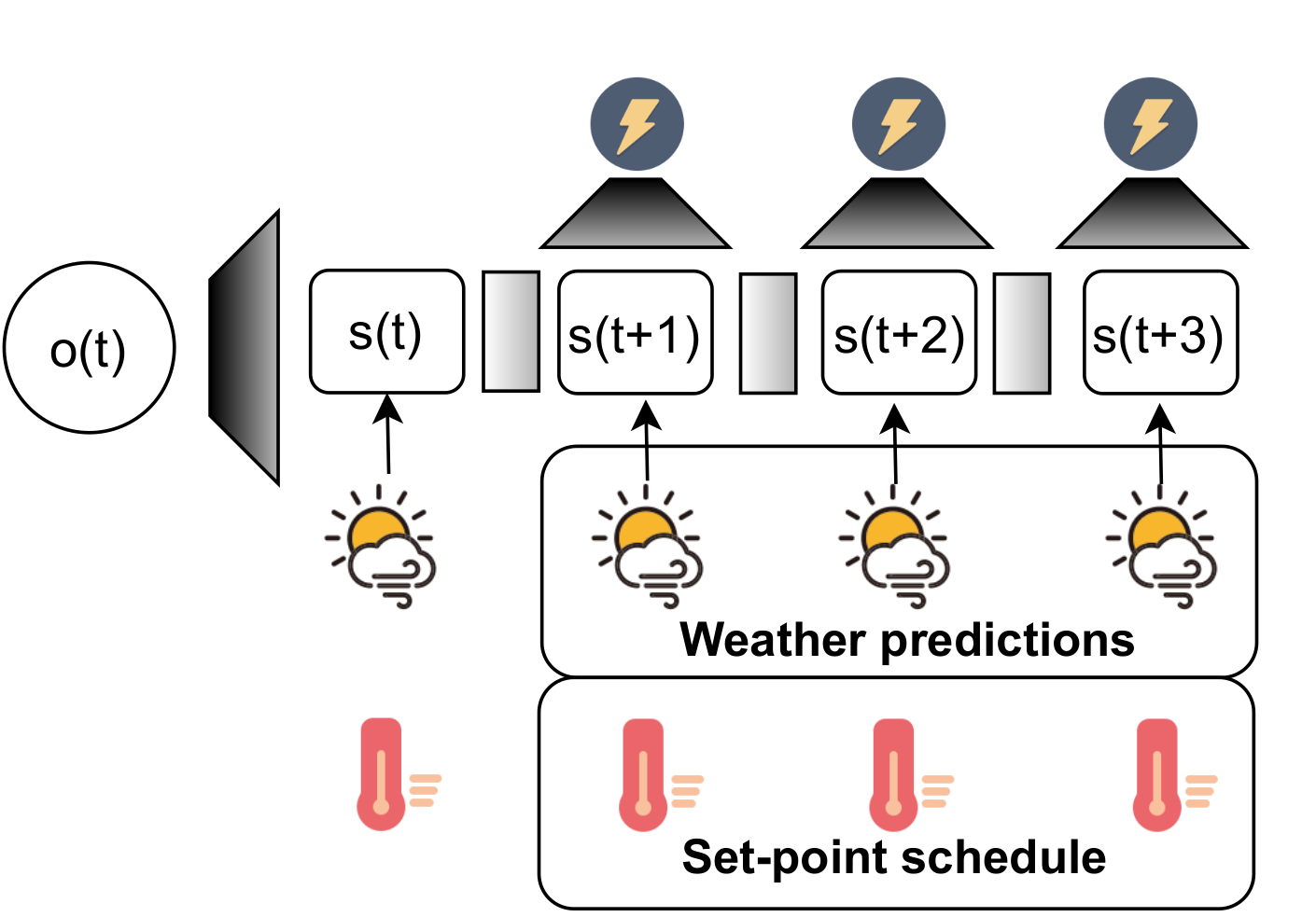}
   \caption{Multi-step ahead power predictions with weather forecasts}
   \label{fig:prediction_scheme}
\end{figure}

\subsection{Distributed Planning Networks Algorithms}

\color{black}
In this section, we present practical methods for solving the LC and coordinator optimization problems introduced in Algorithm \ref{alg:non_convex_admm}. The coordinator problem is a convex penalized least square problem. It can be solved easily with existing solvers. The challenge is more on the side of the LC problem (\ref{opt:local_control}). For this problem, the zones' models are used to plan for the best setpoint. We propose two approaches: the first one is a shooting method, where we differentiate through the prediction model and use the gradient to guide the search; the second one is a random search for the best temperature setpoints. Once the best setpoint changes have been found for the upcoming horizon, only the setpoint changes of the first timestep are implemented. The process is then repeated to allow for re-planning and correction of the control.

The proposed algorithm is called Distributed Planning Networks (DPN) as it embeds Planning Networks in a distributed architecture. Fig. \ref{fig:distributed_planning_networks} and Algorithm \ref{alg:DPN} summarize the approach. 

\subsubsection{Deterministic DPN (DDPN)}

The SSM is fully differentiable. One can formulate the Problem (\ref{opt:local_control}) as an initial value problem and use single shooting to minimize the running cost given by the objective function. The equality constraints representing the dynamics are implicitly satisfied by the prediction model. To handle the inequality constraint (\ref{opt:local_control_c3}), we use projected gradient descent to keep the setpoint changes within the desired interval. This approach assumes continuous values for the setpoints. In practice, we round the values at the resolution of the setpoint tracker of the zone afterward. A stopping criterion on the number of decimals changing in $\Delta_i$ at each iteration is used to make the approach computationally efficient.

\subsubsection{Stochastic DPN (SDPN)}
\label{subsubsec:SDPN}

The random search is conducted as follow:
\begin{enumerate}
    \item Generate random sequences of setpoint changes $\boldsymbol{\Delta}_i$.
    \item Sample the trajectories corresponding to each $\boldsymbol{\Delta}_i$ using the RSSM. For each $\boldsymbol{\Delta}_i$, the RSSM outputs a distribution for the trajectory of states. For one choice of $\boldsymbol{\Delta}_i$, multiple state trajectories are sampled from the distribution to account for uncertainties. The trajectory with the maximum power consumption is kept to optimize on the worst-case scenario.
    \item Select $\boldsymbol{\Delta}_i$ providing the lowest score as defined by  (\ref{opt:local_control}).
\end{enumerate}

While using ADMM, the steps presented above need to be repeated at each iteration. The second step is computationally expensive as it requires many forward passes in the transition and observation networks. However, only the third step depends on the ADMM iteration. To limit the computational time, the trajectories (i.e. the two first steps) are pre-computed and stored before the start of the ADMM iterations. Note that these two steps can be performed in parallel for each zone. At each ADMM iteration, the score of each trajectory is then computed from the pre-computed trajectory with respect to the current power targets $\bar{\textbf{u}}_i$.

\color{black}

\subsubsection{Maximum power constraint conversion}
\label{sec:maximum_power_constraint_conversion}

As mentioned in Section \ref{sec:problem_formulation}, translating the constraint's parameter $\textbf{P}^{max}$ to $\textbf{P}^{tot}$ (defined in Eq. (\ref{opt:const_problem}) ) is important as the building power consumption will be driven towards $\textbf{P}^{tot}$. To find the right $\textbf{P}^{tot}$, we leverage the prediction model of each zone. We forecast the total power with no setpoint changes ($\Delta = 0^{\circ}C)$ to find the business as usual power consumption ($P^{bu}_t$) and with the minimum setpoint change ($\delta_{t,i} = m_i^{\circ}C$ $ \forall t, i$) to find a lower bound on the power consumption ($P^{lb}_t$). If for all $t$ $P^{bu}_t$ is below $P_t^{max}$, the optimal solution of the Problem (\ref{opt:original_control}) is $\boldsymbol{\Delta} = 0$ and no further computation is required. If for any $t$, $P^{lb}_t$ is above $P_t^{max}$, the Problem (\ref{opt:original_control}) is unfeasible. In the latter case, depending on the energy pricing, one could either reduce as much as possible the temperatures to lower the consumption or take no action to preserve comfort. If for all $t$, $P^{max}_t$ falls between $P^{bu}_t$ and $P^{lb}_t$, we proceed with the DPN algorithm and set $\textbf{P}^{tot} =\min (P^{bu}_t, (1-\nu) P^{max}_t)$, where $\nu$ is a slack parameter. Using a slack is important to satisfy the inequality constraint of Problem (\ref{opt:original_control}) with the formulation of Problem (\ref{opt:const_problem}). In addition, as discussed in the next section, it helps mitigate prediction errors. This constraint conversion procedure is detailed in Algorithm \ref{alg:pmax2ptot}.

\color{black}

\begin{algorithm}[t]
    \caption{Distributed Planning Networks}
    \label{alg:DPN}
    \begin{algorithmic}[1]
        \STATE $N \text{number of zones, } H \text{ prediction horizon}$
        \STATE $\{LC_i\}_{i=1, \ldots, N} \text{ Local Controllers}$ 
        \STATE $\{ST_i\}_{i=1, \ldots, N} \text{ Setpoint trackers}$ 
        \STATE $Agg \text{ coordinator}$ 
        \STATE $Env  \text{ Building Environment}$
        
        \STATE $\{\mathcal{D}\}_{i=1, \ldots, N}  \text{ collect data from } Env$
        
        \STATE $\text{Train } LC_{i=1, \ldots, N} \text{ on } \mathcal{D}_{i=1, \ldots, N}$

        \FOR{each timestep t}
            \STATE $\text{Observe state }s_t. $
            \STATE $\textbf{P}^{max} \gets \text{Maximum power consumption over } H$
            \STATE $\{\textbf{T}^{sp}\}_{i=1,\ldots, N} \gets \text{Temperature setpoints over } H$
            \STATE $\textbf{Run the constraint conversion Algorithm \ref{alg:pmax2ptot}}$
            \IF{DO ADMM}
                \STATE $\text{Initialize } \boldsymbol{\Delta}, \Bar{\textbf{u}}, \{\boldsymbol{\lambda}\}_{i=1, \ldots, N}$
                \WHILE{has not converge}
                    \STATE $\{\Bar{\textbf{u}}\}_{i=1, \ldots, N} \gets Agg(\textbf{P}^{tot}) $
                    \STATE $(\{\textbf{u}\}_{i=1\ldots N}, \{\boldsymbol{\Delta}\}_{i=1\ldots N}) \gets LC_i(\Bar{\textbf{u}}_i, \boldsymbol{\lambda}_i)$
                    \STATE $\boldsymbol{\lambda}_i \gets \boldsymbol{\lambda}_i + \rho (\Bar{\textbf{u}}^{k+1}_i - \textbf{u}^k_i)  $
                \ENDWHILE
            \ENDIF
            \STATE $u_{t,i} \gets ST_i(T^{sp}_i, \delta_{t,i})$
            \STATE $Env.step(u_{t,i})$
            \STATE $\text{update } \{\mathcal{D}\}_{i=1, \ldots, N}$
            \IF{enough timesteps}
                \STATE $\text{Update } LC_i \text{ on } \mathcal{D}_i $
            \ENDIF
        \ENDFOR
    \end{algorithmic}
\end{algorithm}

\begin{algorithm}[t]
    \caption{Constraint conversion from $\textbf{P}^{max}$ to $\textbf{P}^{tot}$}\label{alg:pmax2ptot}
    \begin{algorithmic}[1]
        \STATE $\text{Inputs : current state } s_t \text{ and } \textbf{P}^{max}$
        \STATE $\text{Prediction models }\phi_1, \ldots, \phi_N $
        \STATE $\text{Parameters : } \nu, \{m_i\}_{i=1,\ldots,N}, H$
        \STATE \text{Initialize }$\textbf{P}^{tot} \in \mathbb{R}^H$
        \FOR{i=1, \ldots, N}
            \STATE $\textbf{P}_{i}^{bu} = \phi_i(Q(s_t,\boldsymbol{0}))$
            \STATE $\textbf{P}_{i}^{lb} = \phi_i(Q(s_t,\boldsymbol{\Delta}_{min}))$
        \ENDFOR
        \STATE $\textbf{P}^{bu} = \sum_{i=1}^N \textbf{P}_{i}^{bu}, \text{   }\textbf{P}^{lb} = \sum_{i=1}^N \textbf{P}_i^{lb}$
        \IF{$\textbf{P}^{bu} \leq (1-\nu) \textbf{P}^{max}$}
            \STATE $\delta_{t,i} = 0 \text{   } \forall t, i$
            \STATE $\text{DO ADMM = FALSE}$
        \ELSIF{$\textbf{P}^{lb} > (1-\nu) \textbf{P}^{max}$}
            \STATE  $\delta_{t,i} = m_i \text{   } \forall t, i$
            \STATE $\text{DO ADMM = FALSE}$
        \ELSE
            \STATE $\textbf{P}^{tot} = \min(\textbf{P}^{bu},(1-\nu)\textbf{P}^{max}) $
            \STATE $\text{DO ADMM = TRUE }$
        \ENDIF
        
    \end{algorithmic}
    \label{alg:pmax2ptot}
    
\end{algorithm}


\section{Experiments}
\label{sec:experiments}

\color{black}
\subsection{Description of the test environment}

To test the DPN algorithms, we consider the low-rise apartment building from the set of residential reference buildings developed by the Department of Energy (DOE) \cite{DOEreferencebuildings}. \color{black} The building is modeled in EnergyPlus, a software that provides thermal simulations for buildings with state-of-the-art accuracy. The building is composed of 3 floors of 6 apartments, each apartment being a thermal zone. Apartments are equipped with independent HVAC systems, temperature setpoints schedules, occupancy schedules, and appliance usage schedules. \color{black} In Quebec, the most significant strains on the electricity grid happen during winter, due to cold weather and electric heating. We thus consider winter periods ranging from January $1^{st}$ to March $31^{st}$ using Typical Meteorological Year (TMY) files from the city of Montreal. 

The observation, actions, and disturbances are composed of:

\begin{itemize}
    \item observation: zone temperature and HVAC power
    \item disturbances: outdoor temperature, humidity, direct normal irradiance, hour of the day, day of the week.
    \item action: heating temperature setpoint.
\end{itemize}

For simplicity, we consider having access to perfect weather prediction. This assumption is reasonable because this study only considers short-term predictions on which the weather can be accurately predicted \cite{Wang2020, Miller2018}.

\color{black} 

\subsection{Prediction accuracy of the zones' models}
\label{subsec:pred_accuracy}

\begin{figure}
    \centering
    \begin{subfigure}[b]{0.45\textwidth}
        \centering
        \includegraphics[width=\textwidth]{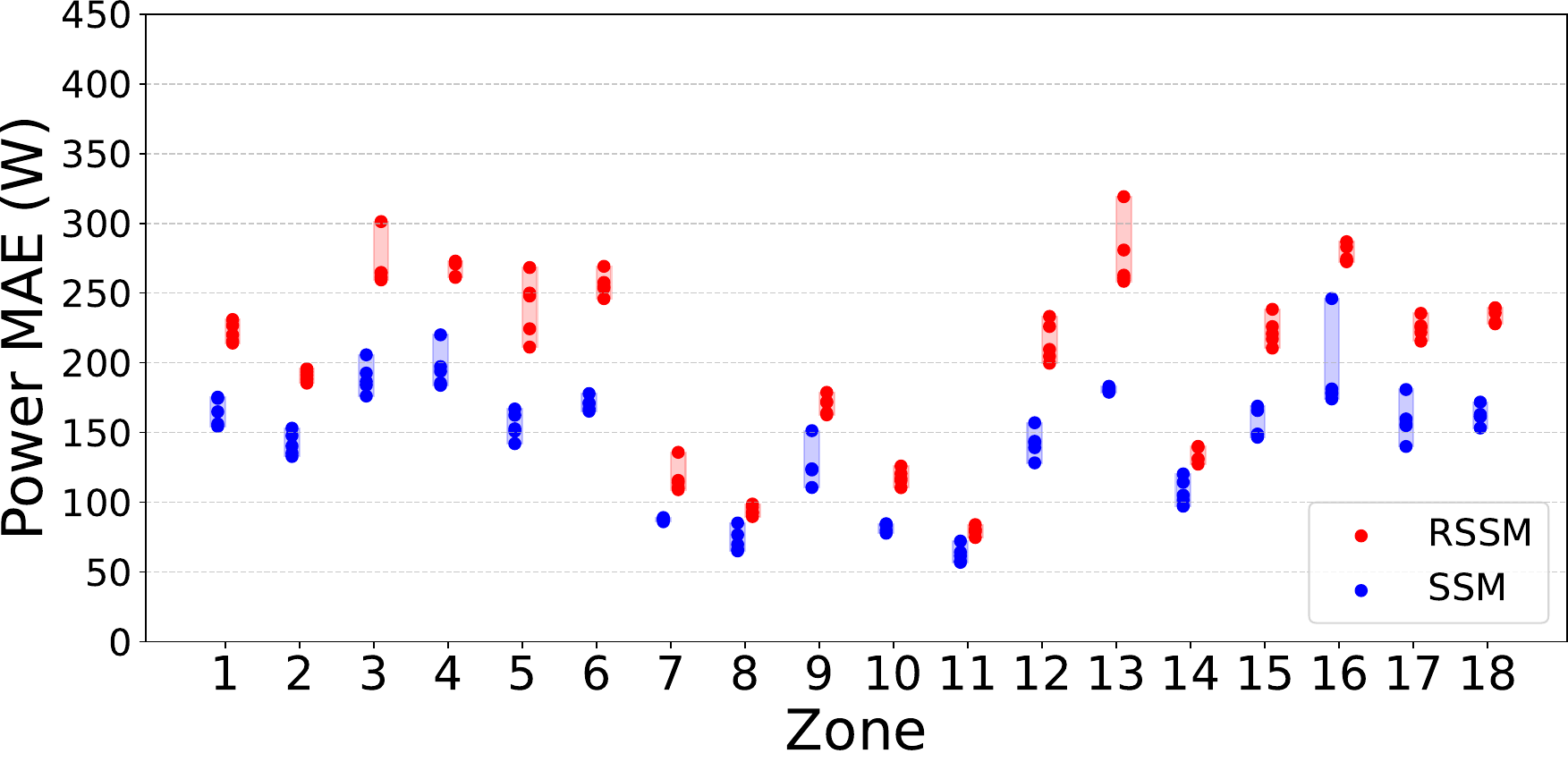}
        \caption{1h ahead predictions}
        \label{fig:scores_1h}
    \end{subfigure}
    \vspace{0.5cm}
    
    \begin{subfigure}[b]{0.45\textwidth}
        \centering
        \includegraphics[width=\textwidth]{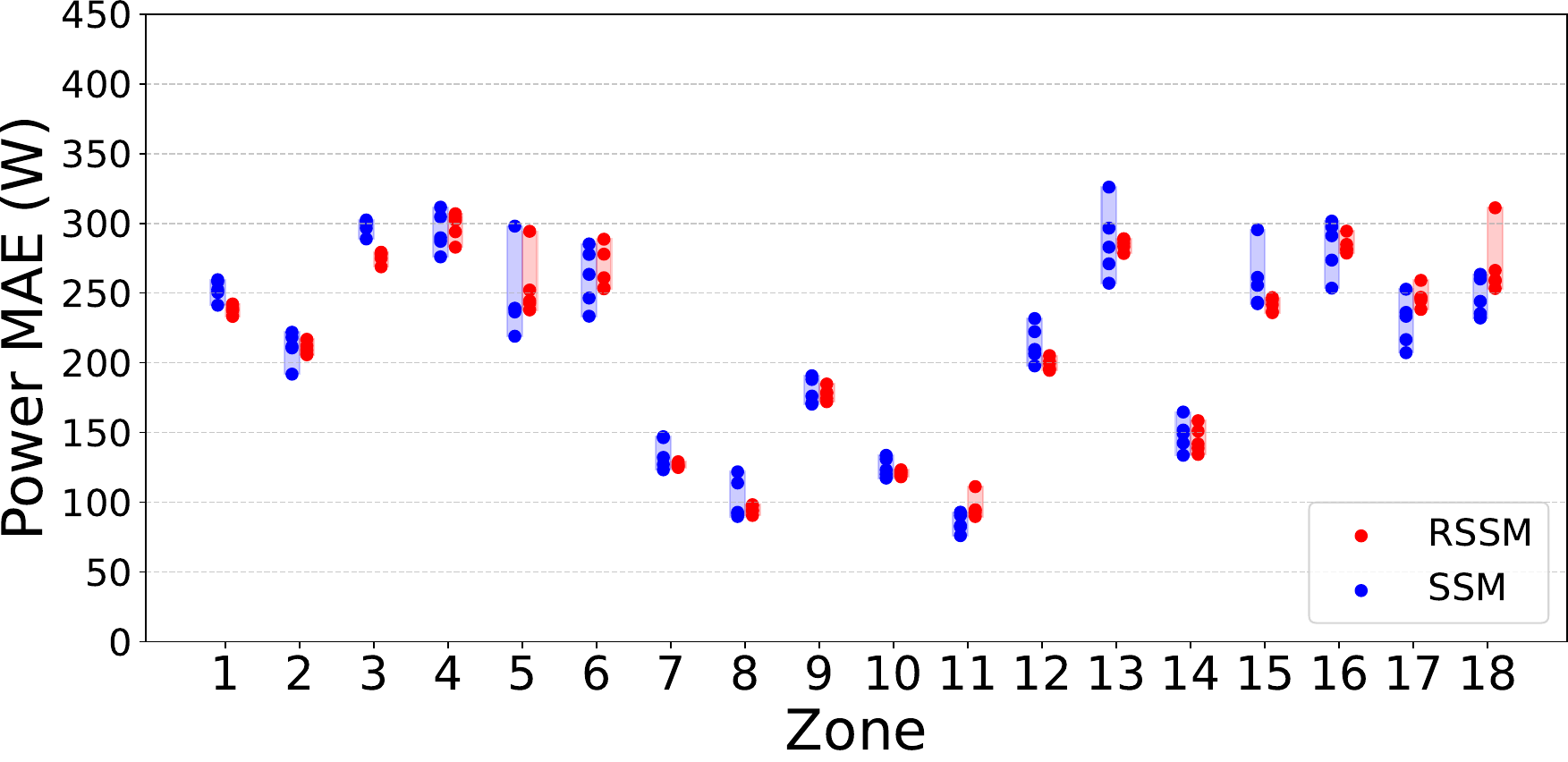}
        \caption{2h ahead predictions}
        \label{fig:scores_2h}
    \end{subfigure}
    \vspace{0.5cm}
    
    \begin{subfigure}[b]{0.45\textwidth}
        \centering
        \includegraphics[width=\textwidth]{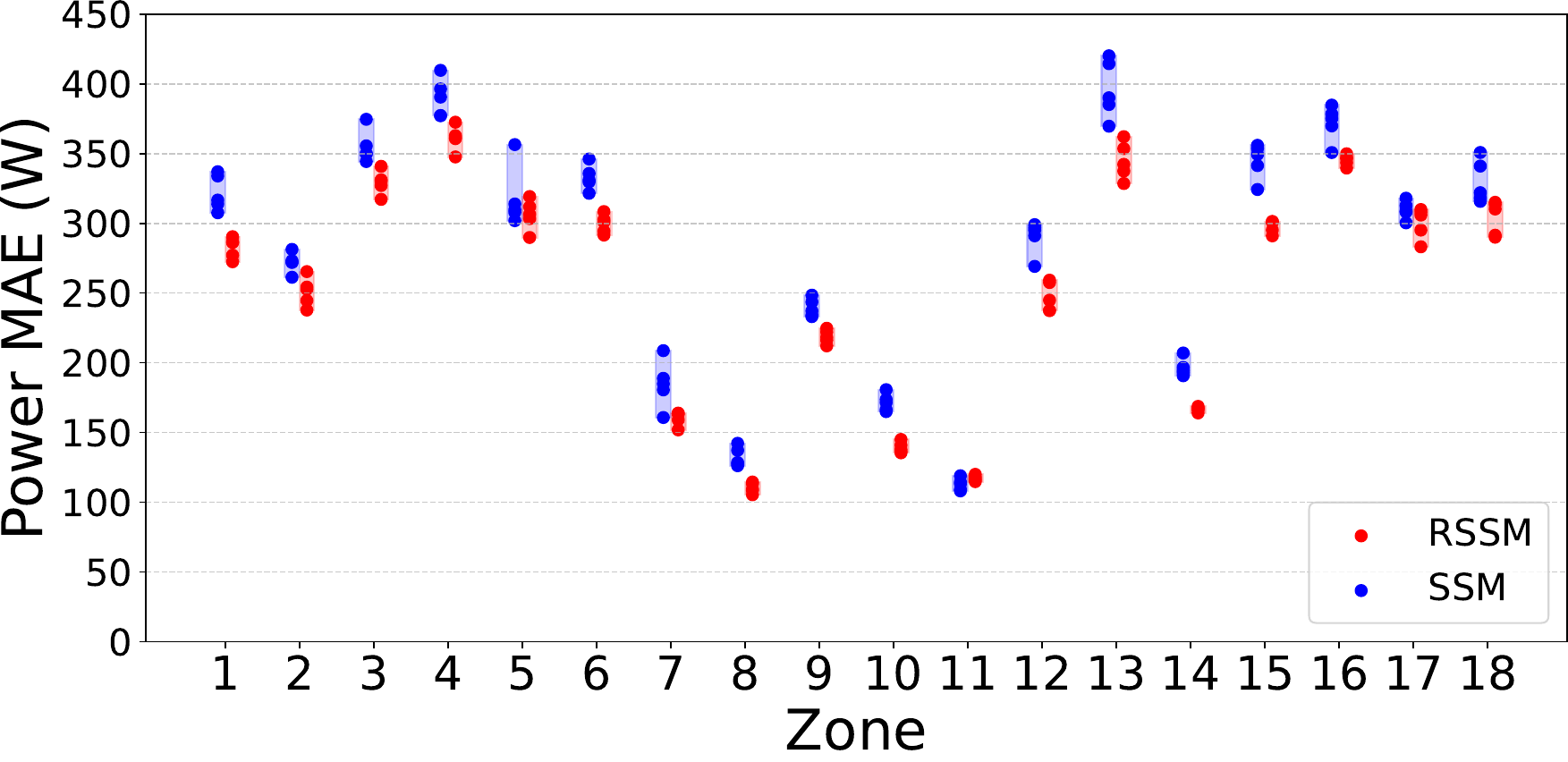}
        \caption{4h ahead predictions}
        \label{fig:scores_4h}
    \end{subfigure}

    \caption{Mean Average Error of the SSM and RSSM in each zone for different horizons. Each model is trained with 5 different training seeds and a dot represents the mean score on the test set for a training seed.}
    \label{fig:model_scores_different_horizons}
\end{figure}

\begin{table}[]
\centering
\small
\begin{tabular}{lccc|ccc}
\toprule
\multirow{2}{*}{Horizon} & \multicolumn{3}{c|}{SSM} & \multicolumn{3}{c}{RSSM} \\
                         & mean & min & max & mean  & min & max \\ 
\midrule
1h                       &     \textbf{4.16}   &     3.94  &     4.30   &   6.84    &  6.45&       7.19   \\
2h                       &     \textbf{6.88}      &     6.54     &    7.10      &     7.29      &   7.00       &     7.69     \\
4h                       &     9.70      &   9.19       &   10.05       &     \textbf{8.92}    &   8.41       &     9.78     \\
\bottomrule
\end{tabular}
\caption{Mean Average Percentage Error (\%) on the entire building HVAC power consumption. Predictions are made by aggregating the predictions of each zone. The mean, minimum and maximum are taken other the 5 training seeds.}
\label{tab:entire_building_prediction_score}
\end{table}

\begin{figure}[b]
    \centering
    \includegraphics[width=0.5\textwidth]{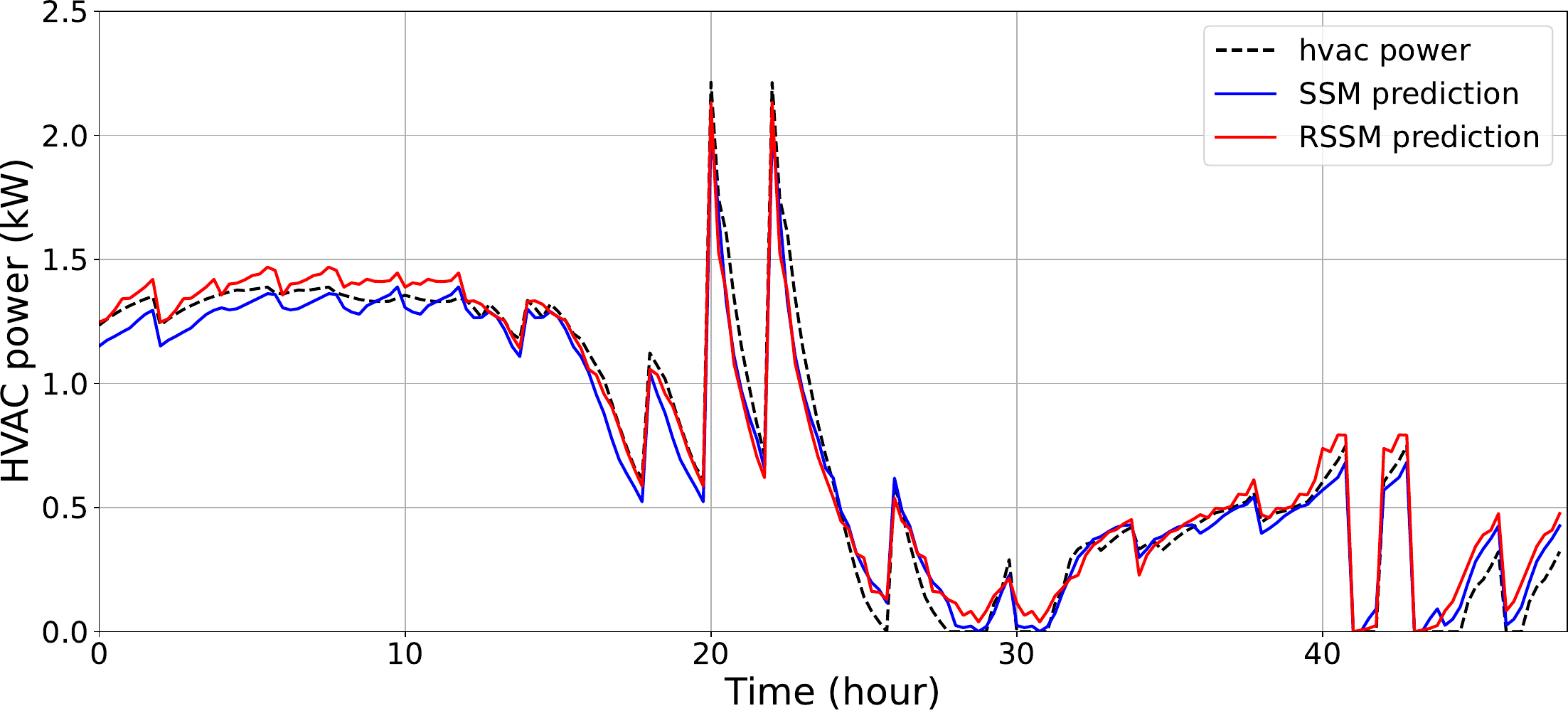}
    \caption{Example of power predictions compared to the actual power consumption for one zone.}
    \label{fig:prediction_accuracy}
\end{figure}

\color{black}

To train the models, data are collected while temperature setpoints vary randomly between $17^\circ C$ and $23^\circ C$ in the apartments. The simulation timestep is set to $15min$ and a setpoint is kept constant for a minimum duration of $1h$ and a maximum duration of $48h$. The training set is composed of 6 months of winter data. A validation set collected during the month of January is used to select the models' best configuration among different model sizes and learning rates. Models are then evaluated on test data collected in February. 

Fig. \ref{fig:model_scores_different_horizons} shows the Mean Average Error (MAE) of the models for different prediction horizons. For 1h ahead predictions, the SMMs perform better than the RSSMs. However, as the horizon lengthens the RSSMs give more accurate predictions than the SMMs. Note that the magnitude of the score for the models of zones 7 to 11 is lower because these zones are on the middle floor and require less heating. As illustrated in Fig. \ref{fig:prediction_accuracy}, both SSMs and RSSMs give accurate predictions of power consumption. Specifically, spikes in the power consumption corresponding to setpoint changes are captured by both model types. Note that these predictions are made without observing the states of the neighboring zones. 

The models' predictions can be aggregated to forecast the entire building's HVAC power and Table \ref{tab:entire_building_prediction_score} shows the Mean Average Percentage Error for both model types. As presented in the next section, the aggregated prediction is crucial to verify if the building's HVAC power satisfies the inequality constraint.

\subsection{DPN control perforamances}

During winter in Quebec, peaks mostly occur in the morning between 6 AM and 9 AM and in the evening between 4 PM and 8 PM \cite{HydroQuebecDR}. To test the DPN algorithm we focus on the morning periods and consider DR events every morning of the test period introduced in Section \ref{subsec:pred_accuracy}. These events translate in a maximum power constraint of $P^{max}_t = 21kW$ between $6AM$ and $9AM$. To satisfy this constraint, the algorithm must reduce the load by $0\%$ to $25\%$, depending on the day. This wide range of power reductions allows us to investigate the ability of the algorithm to handle small adjustments and significant modifications of the zones' temperature. With this setting a total of 10 days out of the 28 days contained events requiring the algorithm to take actions for peak-load reductions. For each zone, we set $-2^\circ C \leq \delta_t \leq 0^{\circ} C$ with increments of $0.25^{\circ}C$. The simulation timestep is $15min$ and re-planning occurs every two timesteps to let the zones' temperature vary after setpoint changes. Note that for the Stochastic DPN, the experiment are repeated with $5$ different computation seeds. The prediction horizon is set to 1 hour, resulting in 81 choices for setpoint changes. With this low-dimension action space, all possible actions are forecasted. For each chosen trajectory of setpoint changes, 100 state trajectories are sampled from the RSSM distribution.

\subsubsection{Control performance}

\begin{table*}[]
\centering
\resizebox{\textwidth}{!}{
\begin{tabular}{l|ccc|cccccc|ccc|cccccc}
\toprule
\multicolumn{1}{c|}{\multirow{2}{*}{Date}} & \multicolumn{9}{c|}{Constraint Violation}                                   & \multicolumn{9}{|c}{Predicted Constraint Violation}                        \\ 
\multicolumn{1}{c|}{}                      & \multicolumn{3}{c|}{DDPN} & \multicolumn{6}{c|}{SDPN} & \multicolumn{3}{c|}{DDPN} & \multicolumn{6}{c}{SDPN} \\
                                           & 0\%   & 5\%   & 10\%  & 1  & 2 & 2 ($\nu=5\%$) & 3  & 4  & 5  & 0\%   & 5\%   & 10\%  & 1  & 2 & 2 ($\nu=5\%$) & 3  & 4  & 5  \\ \midrule
02/04                                      & 83.3 & 41.7 & 0.0 & 0.0 & 0.0 & 0.0 &16.7 & 16.7 & 16.7 & 50.0 & 0.0 & 0.0 & 75.0 & 50.0 & 0.0 &66.7 & 33.3 & 83.3 \\
02/16                                      & 33.3 & 16.7 & 0.0 & 0.0 & 0.0 & 0.0 &0.0 & 0.0 & 0.0 & 8.3 & 0.0 & 0.0 & 41.7 & 33.3 & 0.0 &25.0 & 33.3 & 25.0 \\
02/17                                      & 66.7 & 16.7 & 0.0 & 0.0 & 0.0 & 0.0 & 0.0 & 0.0 & 0.0 & 41.7 & 0.0 & 0.0 & 41.7 & 50.0 & 0.0 &41.7 & 33.3 & 58.3 \\
02/18                                      & 66.7 & 25.0 & 8.3 & 16.7 & 16.7 & 0.0 & 0.0 & 0.0 & 0.0 & 41.7 & 0.0 & 0.0 & 50.0 & 58.3 &0.0 & 66.7 & 58.3 & 58.3 \\
02/21                                      & 58.3 & 8.3 & 0.0 & 0.0 & 0.0 & 0.0 &0.0 & 0.0 & 0.0 & 58.3 & 0.0 & 0.0 & 33.3 & 58.3 & 0.0& 58.3 & 50.0 & 41.7 \\
02/22                                      & 83.3 & 25.0 & 8.3 & 0.0 & 0.0 & 0.0 &0.0 & 8.3 & 8.3 & 41.7 & 0.0 & 0.0 & 50.0 & 83.3 &8.3 &58.3 & 66.7 & 66.7 \\
02/23                                      & 75.0 & 25.0 & 8.3 & 0.0 & 8.3 & 0.0 & 0.0 & 0.0 & 8.3 & 58.3 & 0.0 & 0.0 & 58.3 & 58.3 & 8.3&66.7 & 58.3 & 58.3 \\
02/24                                      & 33.3 & 16.7 & 0.0 & 0.0 & 0.0 & 0.0 & 0.0 & 0.0 & 0.0 & 41.7 & 0.0 & 0.0 & 33.3 & 33.3 & 0.0&25.0 & 33.3 & 41.7 \\
02/26                                      & 75.0 & 41.7 & 8.3 & 0.0 & 0.0 & 0.0 & 0.0 & 0.0 & 0.0 & 50.0 & 0.0 & 0.0 & 58.3 & 33.3 & 0.0&33.3 & 33.3 & 33.3 \\
02/28                                      & 25.0 & 0.0 & 0.0 & 0.0 & 25.0 & 0.0 & 8.3 & 25.0 & 0.0 & 0.0 & 0.0 & 0.0 & 8.3 & 0.0 & 0.0 &8.3 & 0.0 & 8.3 \\
\bottomrule
\end{tabular}}

\caption{\textcolor{black}{Percentage of timesteps where the constraint is not satisfied, and not satisfied in prediction during planning. For DDPN, results are shown for different slack values $\nu$. For the SDPN, the slack is set to $\nu=0\%$ and results are displayed for each computation seed.}}
\label{tab:constraint_violation}
\end{table*}

\begin{figure}
\centering
\begin{subfigure}[b]{0.49\textwidth}
    \includegraphics[width=\textwidth]{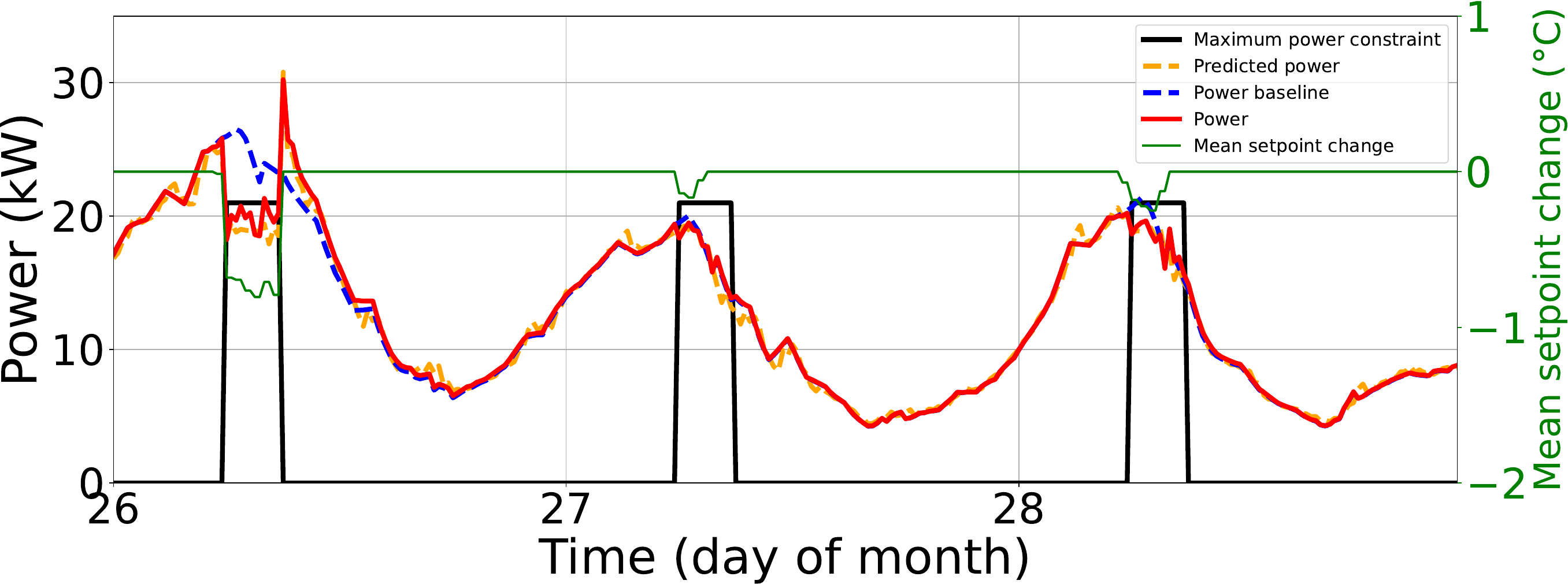}
    \caption{Determinisitc DPN ($\nu = 10\%)$}
    \label{fig:ddpn_control_results}
\end{subfigure}
\vspace{0.5cm}

\begin{subfigure}[b]{0.49\textwidth}
    \includegraphics[width=\textwidth]{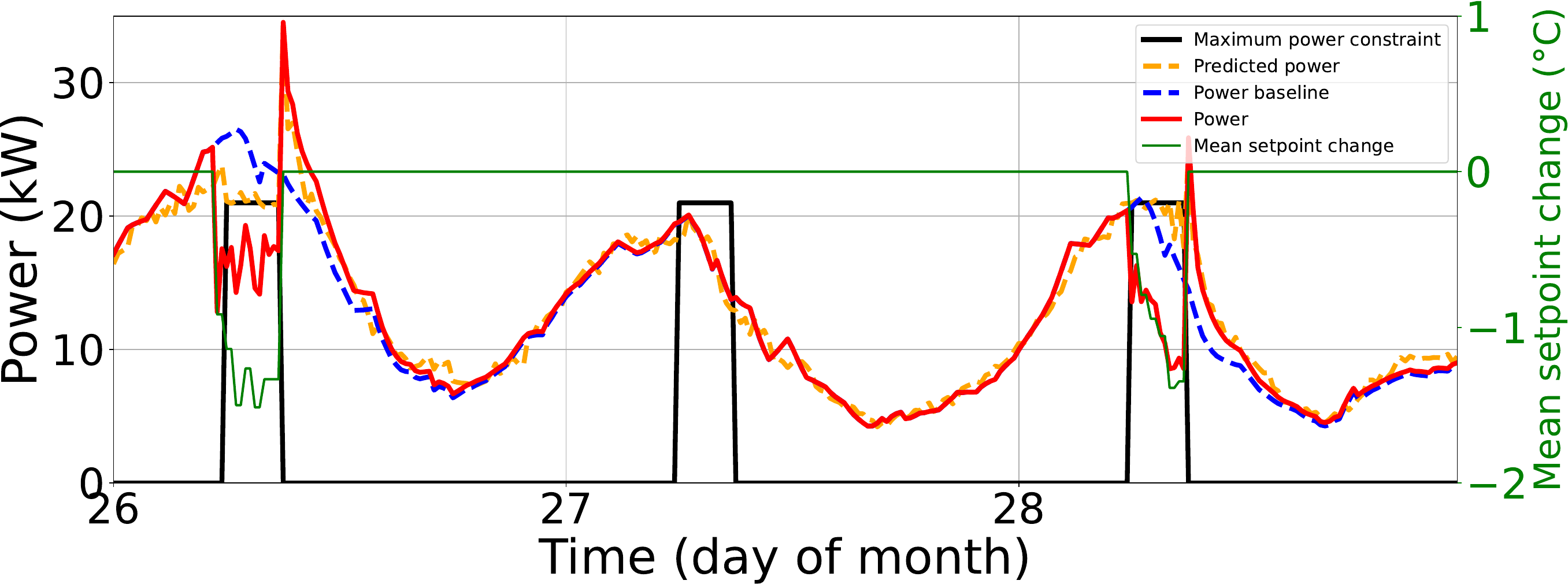}
    \caption{Stochastic DPN}
    \label{fig:sdpn_control_results}
\end{subfigure}
\caption{Three days control example for the Deterministic and Stochastic DPN. For simplicity, the results of a single seed are displayed for the SDPN.}
\label{fig:ddpn_sdpn_control_results}
\end{figure}

\begin{figure}
    \centering
    \includegraphics[width=0.49\textwidth]{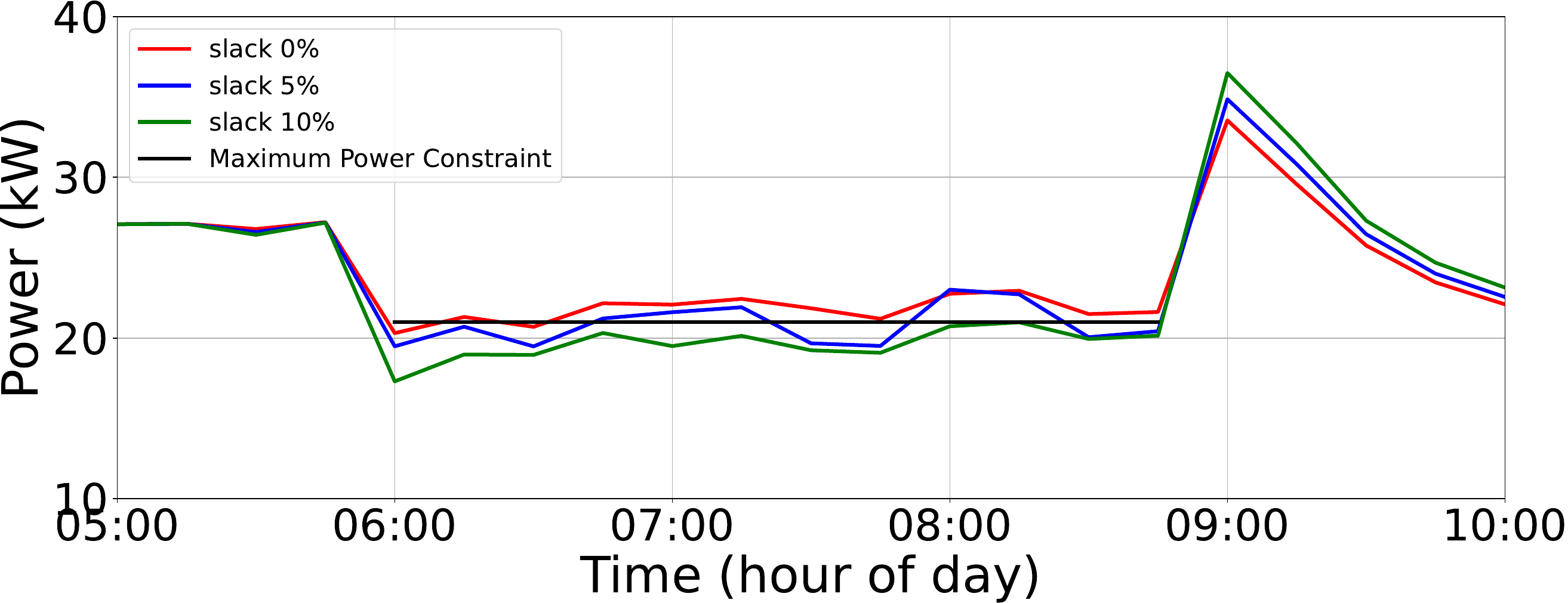}
    \caption{Effect of the slack parameter $\nu$ of the Deterministic DPN control.}
    \label{fig:ddpn_slakc_comparison}
\end{figure}

\begin{figure}
    \centering
    \includegraphics[width=0.49\textwidth]{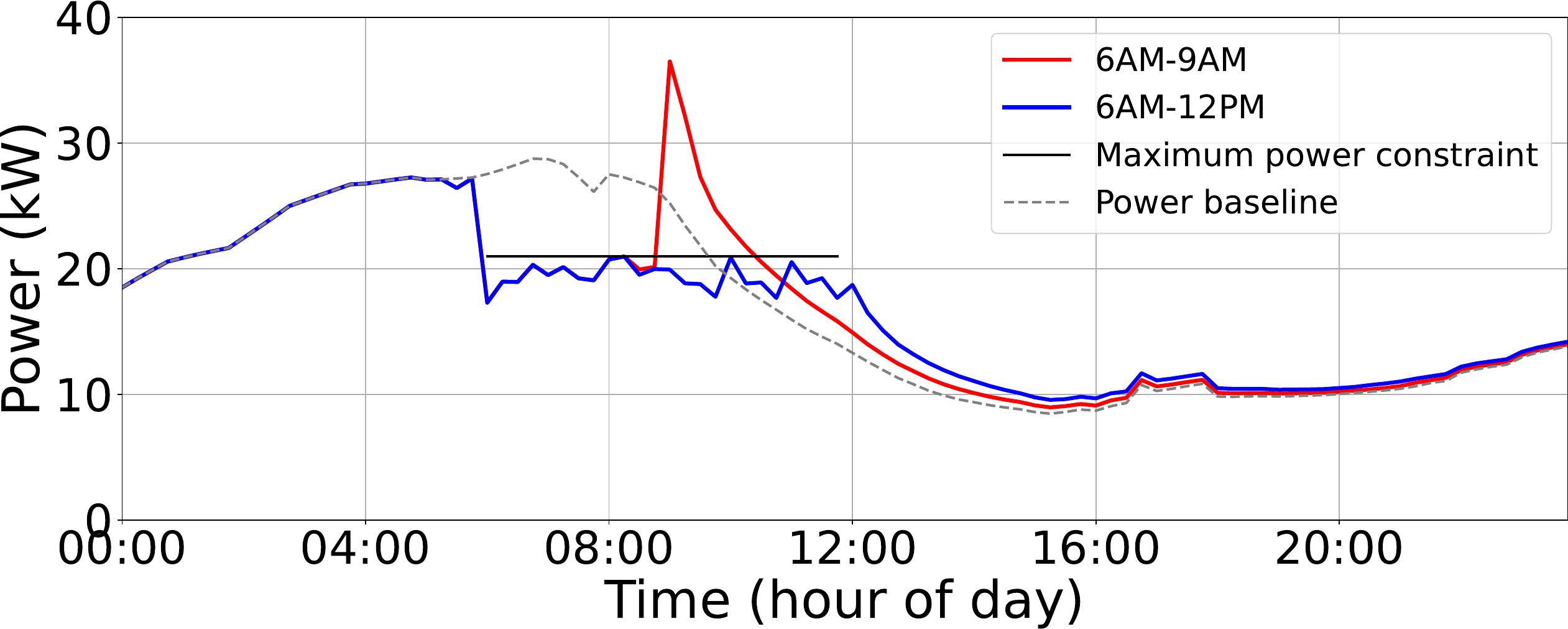}
    \caption{Mitigation of the rebound peak for the Deterministic DPN algorithm. For the red curve the event ends at 9 AM while for the blue curve it ends at 12 AM.}
    \label{fig:ddpn_rebound_peak}
\end{figure}

\begin{table*}
\centering
\resizebox{\textwidth}{!}{
\begin{tabular}{l|l|l|l|l|l|l|l|l|l|l}
\hline
 \multirow{10}{*}{\rotatebox[origin=c]{90}{Deterministic DPN}} & zone       & mean $(^\circ C)$   &  $\delta=0$     & $ 0 <\delta \leq 1$  & $ 1 <\delta \leq 2$    & zone        & mean $(^\circ C)$  & $\delta=0$    & $ 0 <\delta \leq 1$  & $ 1 <\delta \leq 2$   \\ \cline{2-11}
 & \textbf{1} & -0.72 & 38     & 20 & 42    & \textbf{10} & -0.50 & 39     & 39  & 21 \\ 
& \textbf{2} & -0.47  & 45     & 32 & 23 & \textbf{11} & -0.54 & 33 & 39  & 27  \\ 
& \textbf{3} & -0.72  & 33     & 29  & 38 & \textbf{12} & -0.71 & 41     & 20 & 39  \\ 
& \textbf{4} & -0.86  & 27  & 30  & 42 & \textbf{13} & -0.89 & 35     & 9  & 56 \\ 
& \textbf{5} & -0.58  & 47     & 21 & 32 & \textbf{14} & -0.37 & 39 & 52  & 9 \\ 
& \textbf{6} & -0.44  & 45     & 26  & 29 & \textbf{15} & -0.58 & 45     & 26  & 29 \\ 
& \textbf{7} & -0.48  & 36  & 41 & 23 & \textbf{16} & -0.73 & 23     & 39  & 38  \\ 
& \textbf{8} & -0.52  & 35 & 42 & 23 & \textbf{17} & -0.81 & 35     & 15 & 50 \\ 
& \textbf{9} & -0.81  & 33 & 21  & 45 & \textbf{18} & -0.55 & 36     & 36 & 27 \\ \hline
\multirow{10}{*}{\rotatebox[origin=c]{90}{Stochastic DPN}} & zone       & mean $(^\circ C)$   &  $\delta=0$     & $ 0 <\delta \leq 1$  & $ 1 <\delta \leq 2$    & zone        & mean $(^\circ C)$  & $\delta=0$    & $ 0 <\delta \leq 1$  & $ 1 <\delta \leq 2$   \\ \cline{2-11}
 & \textbf{1} & -1.14 & 12     & 20 & 68    & \textbf{10} & -1.11 & 12   & 17  & 71 \\ 
& \textbf{2} & -1.25  & 14     & 14 & 73 & \textbf{11} & -0.92 & 12 & 33  & 55  \\ 
& \textbf{3} & -1.18  & 11     & 23  & 67 & \textbf{12} & -1.23 & 11     & 12 & 76  \\ 
& \textbf{4} & -1.12  & 15  & 18  & 67 & \textbf{13} & -1.23 & 12     & 20  & 70 \\ 
& \textbf{5} & -1.06  & 15     & 20 & 65 & \textbf{14} & -1.21 & 12 & 15  & 73 \\ 
& \textbf{6} & -1.20  & 12     & 14  & 74 & \textbf{15} & -1.22 & 12    & 18  & 70 \\ 
& \textbf{7} & -1.13  & 15  & 18 & 67 & \textbf{16} & -1.33 & 11     & 09  & 80  \\ 
& \textbf{8} & -1.13  & 14 & 09 & 77 & \textbf{17} & -1.25 & 12    & 15 & 73 \\
& \textbf{9} & -1.12  & 15 & 18  & 67 & \textbf{18} & -1.24 & 11     & 12 & 77 \\
\hline
\end{tabular}}
\caption{Setpoint changes distribution for each zone during the DR events requiring actions. For the SDPN, results are consistent across all the seeds, for simplicity the results of a single seed are presented. For the DDPN, $\nu = 10\%$.}
\label{tab:dr_temperatures}
\end{table*}

Table \ref{tab:constraint_violation} shows the percentage of timesteps where the power constraint was not satisfied for each of the DR events requiring actions. When the slack parameter $\nu = 0 \%$ the algorithm aims at consuming exactly $P^{max}$ when the predicted business-as-usual power exceeds $P^{max}$, which often leads to constraint violations. We observe that increasing the slack makes the DDPN algorithm aim below the constraint, which decreases the percentage of constraint violation. It is interesting to notice that for the DDPN with slacks of $5\%$ and $10\%$, the constraints are always satisfied in predictions. It means that the control algorithm is working perfectly and constraint violations are due to prediction errors only. Increasing the slack helps mitigate the prediction errors but also increases the power reduction, thus impacting the comfort. A balance must be made between constraint satisfaction and comfort. Fig. \ref{fig:ddpn_slakc_comparison} illustrates the effect of the slack parameter on the total power consumption. For the SDPN performances, the optimization is based on the worst-case scenario and $\nu$ is set to $0\%$ to avoid excessive conservatism. The algorithm thus aims at consumption $P^{max}$, which explains the high number of constraint violations in prediction. Despite optimizing based on the worst case, the algorithm still aims at consuming exactly the maximum power available, which may lead to constraint violations. This can easily mitigated by using non-zero slack. To illustrate this, we re-run the second seed that has the most constraint violation with a slack of $5\%$. The column "$2$ $(\nu=5\%)$" in Table \ref{tab:constraint_violation} shows that every event was satisfied with this setting.

Fig. \ref{fig:ddpn_sdpn_control_results} compares the actual power consumption and the predicted power consumption using the DPN algorithms against the baseline power consumption from the built-in EnergyPlus controller. On the first day displayed, the predicted power consumption is around $(1-\nu) P^{max}$ for the DDPN algorithm. However, for one timestep, we observe that the consumed power exceeds the maximum. This one-timestep prediction error can happen even with excellent average prediction performances and illustrates the importance of prediction accuracy. On the same day, we see that the SDPN algorithm takes greater setpoint reductions to satisfy the constraint. It is important to notice that our problem formulation with a hard constraint on the power is the most challenging setting because it implies extremely accurate predictions to avoid an over-conservative algorithm. That said, depending on the pricing, a hard constraint is not always required. Since the DPN algorithm is based on the optimization problem (\ref{opt:const_problem}) it may also be used as is to effectively limit the power consumption without a hard constraint on the power. On the third day displayed in Fig. \ref{fig:ddpn_sdpn_control_results}, we observe the ability of DDPN to slightly adjust the temperature to satisfy the constraint. The SDPN algorithm is more conservative and takes greater reductions. One possible drawback of using a large slack value is to take unnecessary actions to lower the consumption, as shown on the second day in Fig. \ref{fig:ddpn_control_results}. As illustrated in Fig. \ref{fig:ddpn_sdpn_control_results}, a rebound peak may occur after the DR event because the setpoints are increased at the same time in every zone. This can be mitigated by increasing the duration of the power constraint to restart the heating systems later in the day when the total power consumption is lower, as shown in Fig. \ref{fig:ddpn_rebound_peak}.

Regarding the setpoint changes, Table \ref{tab:dr_temperatures} shows that each zone participates in the power reduction effort. While there is nothing explicitly enforcing fairness in our algorithm, the similar mean setpoint changes across the zones may be explained by the uniform initialization of ADMM, which first sets a setpoint change of $-1^\circ C$ in each zone. We also observe that the setpoint changes are larger with SDPN because of the robust optimization. 

\subsubsection{DPN convergence}

Fig. \ref{fig:admm_convergence} shows the mean and standard deviation of the residuals, taken over the first 20 iterations of each ADMM usage. In both cases, the residuals are decreasing and at each iteration the algorithm converges to a solution. Concerning computational time, resolving the LC problems consumes the most time. As shown in Table \ref{tab:computation_time}, for DDPN, LC computation accounts for nearly 98\% of an iteration duration. This process is conducted in parallel for each zone and, with sufficient computing resources, is independent of the number of zones, making the algorithm scalable to the size of the building. We also notice that the SDPN is significantly faster. This is because the action space is low-dimensional and state trajectories are precomputed, as explained in Section \ref{subsubsec:SDPN}. That said, considering longer horizons or larger comfort intervals would increase exponentially the action space dimension. Exploring larger action spaces randomly becomes rapidly computationally intractable, unlike the gradient-based search used in DDPN.

\begin{figure}
\centering
\begin{subfigure}[b]{0.49\textwidth}
    \includegraphics[width=\textwidth]{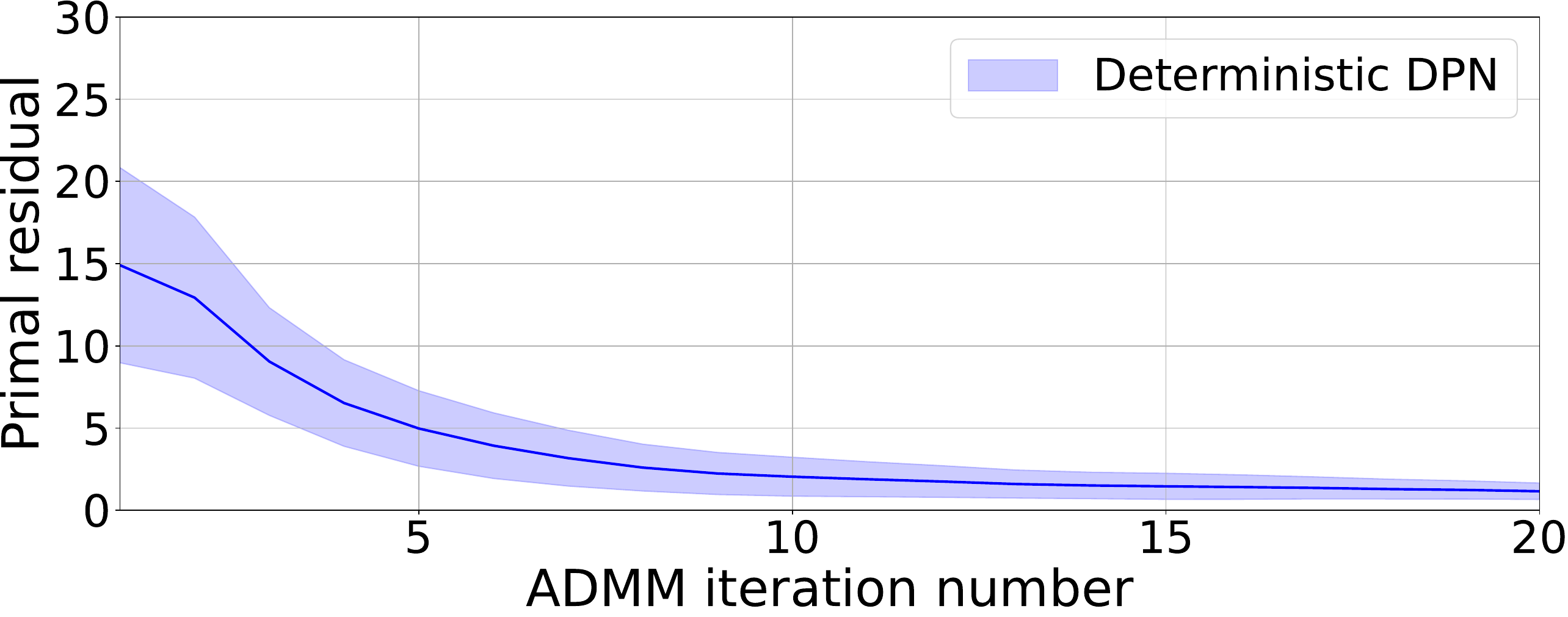}
    \caption{DDPN}
    \label{fig:admm_convergence_ddpn}
\end{subfigure}
\vspace{0.5cm}

\begin{subfigure}[b]{0.49\textwidth}
    \includegraphics[width=\textwidth]{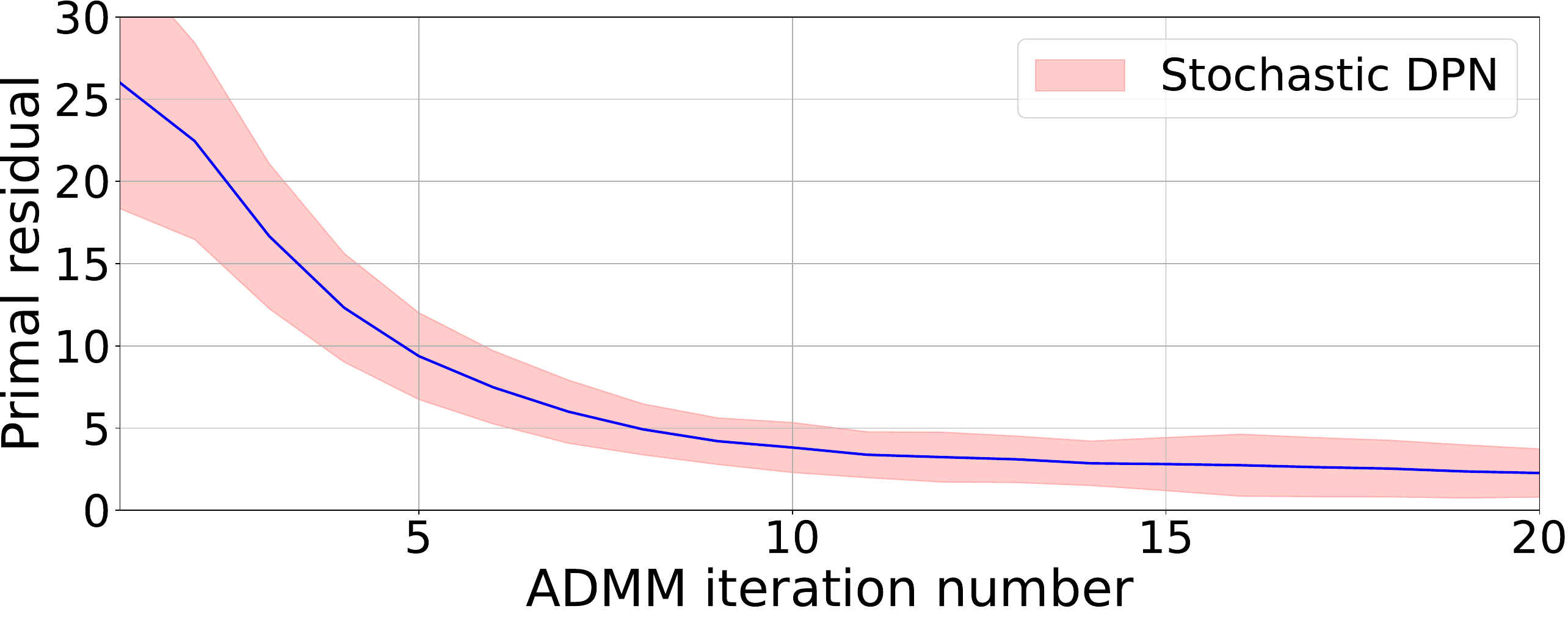}
    \caption{SDPN}
    \label{fig:admm_convergence_sdpn}
\end{subfigure}
\caption{Deterministic and Stochastic DPN convergence}
\label{fig:admm_convergence}
\end{figure}

\begin{table}
\centering
\begin{tabular}{l|l|l}
\toprule
                    & Deterministic DPN & Stochastic DPN \\ \midrule
DPN call   &    160.0 $\pm$ 42.04  &   14.43 $\pm$ 0.39    \\ 
Coordinator iteration    &   0.11 $\pm$ 0.02    &   0.09 $\pm$ 0.01    \\ 
LC iteration  &    3.05  $\pm$ 0.79    &    0.03 $\pm$ 0.0  \\
\bottomrule
\end{tabular}
\caption{Mean $\pm$ std of the computation time in seconds for the entire DPN algorithm, the coordinator and local controllers problem. Note that each local controller was running on a 2-core CPU with 4GB of RAM.}
\label{tab:computation_time}
\end{table}

\color{black}


\section{Discussion and further work}
\label{sec:discussion_and_further_work}

The results presented above demonstrate the effectiveness of the algorithm to distribute a global coupled power constraint to local heating systems in a multi-zone building.

The experiments have revealed that a low prediction error is key to the success of the method. Prediction errors are not an issue if one aims at maintaining the power around a target value. However, most of the constraints implying a maximum power are hard constraints and an inaccurate forecast will cause the constraint to be violated. \textcolor{black}{The maximum prediction error could easily exceed the mean error of $4.16\%$ or $6.84\%$ reached by our models, especially after setpoint changes when the power consumption varies sharply}. To improve the accuracy of the models, one may include the state of the neighboring zones in the observations of a zone. Another possibility is to retrain the models on data collected during the control period. Such data are more representative of encountered states than the data coming from random control of the temperature setpoints used to first train the models.
\vspace{0.2cm}

An interesting extension of this work would be to measure thermal comfort based on predictions of the temperature of each zone instead of on the temperature setpoints. The local objectives $g$, defined in Eq. (\ref{opt:sharing_problem}), would become 

$$g(\boldsymbol{\Delta}_i) = \norm{T_{t,i} - T^{sp}_{t,i}}_2^2,$$

where $T_{t,i}$ is the temperature in zone $i$ at time $t$. \textcolor{black}{Such a formulation would permit pre-heating strategies and further enhance energy conservation.} However, it complicates the theoretical analysis because such a $g$ function is no longer convex. In addition, some preliminary experiments suggest that it is hard to accurately predict the temperature with fully data-driven models. If the temperature setpoint is passed as an input, the models often learn to predict the setpoint as the temperature forecast. The use of physics-informed models may be necessary to ensure accurate temperature predictions \cite{Drgona2021, Djeunou2022}. \textcolor{black}{In addition, our experiments showed that relaxing the inequality constraint of Problem (\ref{opt:original_control}) with the symmetric term $\displaystyle ||A\boldsymbol{u} - \boldsymbol{P}^{tot}||_2^2$ can ensure the constraint satisfaction given the right choice of $\bold{P}^{tot}$ and normalized variables $\boldsymbol{u}$. For the algorithm to converge, the function $\ell$ needs to have a Lipschitz continuous gradient so an interesting research direction is to use an asymmetric function such as a log-barrier to relax the inequality constraint. The convergence proof will still hold and it could remove the need for a slack parameter. However, the coordinator problem will no longer be a simple penalized least-square problem and will be harder to solve.}

We trained all the prediction models from scratch. However, there are many similarities from one zone to the other. An interesting improvement for this work will be to leverage the training of a model in one zone to accelerate the training of other models. Existing techniques of Meta-Learning could be used for this purpose. This will greatly improve the generalization of the approach as models from some buildings could be used for new buildings.

Finally, tests are all performed using a virtual environment. Even though EnergyPlus provides state-of-the-art accuracy for building simulation, future works should include tests on real buildings to further validate the method.


\section{Conclusion}

In this work, we present a distributed algorithm that combines ADMM and Deep Learning models to control the temperature setpoints and act on the power consumption of HVAC systems in buildings. Based on a global coupled power constraint, the algorithm distributes the power across the Local Controller systems. This approach does not assume any particular type of local heating or cooling system and learns its model entirely from data. We provided a theoretical framework and derived a practical algorithm that we implemented to test the efficiency of the approach. Tests are performed on a large residential building composed of 18 zones, modeled using EnergyPlus. By combining the zones' prediction models, we reached a mean error of $4.16\%$ with SSM and $6.84\%$ with RSSM on the prediction of the total heating power consumption for the $1h$ ahead horizon used for planning. Both the deterministic and stochastic versions of the Distributed Planning Networks algorithm are tested on a power control task during DR events. The power is successfully reduced below a given maximum power limit and the effort is shared across all zones. 

\bibliographystyle{./bibliography/IEEEtran}

\begin{thebibliography}{10}
\providecommand{\url}[1]{#1}
\csname url@samestyle\endcsname
\providecommand{\newblock}{\relax}
\providecommand{\bibinfo}[2]{#2}
\providecommand{\BIBentrySTDinterwordspacing}{\spaceskip=0pt\relax}
\providecommand{\BIBentryALTinterwordstretchfactor}{4}
\providecommand{\BIBentryALTinterwordspacing}{\spaceskip=\fontdimen2\font plus
\BIBentryALTinterwordstretchfactor\fontdimen3\font minus \fontdimen4\font\relax}
\providecommand{\BIBforeignlanguage}[2]{{%
\expandafter\ifx\csname l@#1\endcsname\relax
\typeout{** WARNING: IEEEtran.bst: No hyphenation pattern has been}%
\typeout{** loaded for the language `#1'. Using the pattern for}%
\typeout{** the default language instead.}%
\else
\language=\csname l@#1\endcsname
\fi
#2}}
\providecommand{\BIBdecl}{\relax}
\BIBdecl

\bibitem{levesque2018}
A.~Levesque, R.~C. Pietzcker, L.~Baumstark, S.~{De Stercke}, A.~Grübler, and G.~Luderer, ``How much energy will buildings consume in 2100? a global perspective within a scenario framework,'' \emph{Energy}, vol. 148, pp. 514--527, 2018.

\bibitem{Fontenot2019}
\BIBentryALTinterwordspacing
H.~Fontenot and B.~Dong, ``Modeling and control of building-integrated microgrids for optimal energy management – a review,'' \emph{Applied Energy}, vol. 254, p. 113689, 2019. [Online]. Available: \url{https://www.sciencedirect.com/science/article/pii/S0306261919313765}
\BIBentrySTDinterwordspacing

\bibitem{Ouammi2020}
A.~Ouammi, Y.~Achour, D.~Zejli, and H.~Dagdougui, ``Supervisory model predictive control for optimal energy management of networked smart greenhouses integrated microgrid,'' \emph{IEEE Transactions on Automation Science and Engineering}, vol.~17, no.~1, pp. 117--128, 2020.

\bibitem{Belussi2019}
\BIBentryALTinterwordspacing
L.~B. et~al., ``A review of performance of zero energy buildings and energy efficiency solutions,'' \emph{Journal of Building Engineering}, vol.~25, p. 100772, 2019. [Online]. Available: \url{https://www.sciencedirect.com/science/article/pii/S235271021831461X}
\BIBentrySTDinterwordspacing

\bibitem{Dafonseca2021}
\BIBentryALTinterwordspacing
A.~L. {da Fonseca}, K.~M. Chvatal, and R.~A. Fernandes, ``Thermal comfort maintenance in demand response programs: A critical review,'' \emph{Renewable and Sustainable Energy Reviews}, vol. 141, p. 110847, 2021. [Online]. Available: \url{https://www.sciencedirect.com/science/article/pii/S1364032121001416}
\BIBentrySTDinterwordspacing

\bibitem{Zhang2023}
\BIBentryALTinterwordspacing
H.~Zhang, F.~Xiao, C.~Zhang, and R.~Li, ``A multi-agent system based coordinated multi-objective optimal load scheduling strategy using marginal emission factors for building cluster demand response,'' \emph{Energy and Buildings}, vol. 281, p. 112765, 2023. [Online]. Available: \url{https://www.sciencedirect.com/science/article/pii/S0378778822009367}
\BIBentrySTDinterwordspacing

\bibitem{Xie2023}
\BIBentryALTinterwordspacing
J.~Xie, A.~Ajagekar, and F.~You, ``Multi-agent attention-based deep reinforcement learning for demand response in grid-responsive buildings,'' \emph{Applied Energy}, vol. 342, p. 121162, 2023. [Online]. Available: \url{https://www.sciencedirect.com/science/article/pii/S0306261923005263}
\BIBentrySTDinterwordspacing

\bibitem{Deconinck2016}
\BIBentryALTinterwordspacing
R.~D. Coninck, F.~Magnusson, J.~Åkesson, and L.~Helsen, ``Toolbox for development and validation of grey-box building models for forecasting and control,'' \emph{Journal of Building Performance Simulation}, vol.~9, no.~3, pp. 288--303, 2016. [Online]. Available: \url{https://doi.org/10.1080/19401493.2015.1046933}
\BIBentrySTDinterwordspacing

\bibitem{Rezaei2020}
E.~Rezaei and H.~Dagdougui, ``Optimal real-time energy management in apartment building integrating microgrid with multizone hvac control,'' \emph{IEEE Transactions on Industrial Informatics}, vol.~16, no.~11, pp. 6848--6856, 2020.

\bibitem{tanaskovic2017}
\BIBentryALTinterwordspacing
M.~Tanaskovic, D.~Sturzenegger, R.~Smith, and M.~Morari, ``Robust adaptive model predictive building climate control,'' \emph{IFAC-PapersOnLine}, vol.~50, no.~1, pp. 1871--1876, 2017, 20th IFAC World Congress. [Online]. Available: \url{https://www.sciencedirect.com/science/article/pii/S2405896317305700}
\BIBentrySTDinterwordspacing

\bibitem{Vazquezcanteli2019}
\BIBentryALTinterwordspacing
J.~R. Vázquez-Canteli and Z.~Nagy, ``Reinforcement learning for demand response: A review of algorithms and modeling techniques,'' \emph{Applied Energy}, vol. 235, pp. 1072--1089, 2019. [Online]. Available: \url{https://www.sciencedirect.com/science/article/pii/S0306261918317082}
\BIBentrySTDinterwordspacing

\bibitem{Zhang2019}
Z.~Zhang, A.~Chong, Y.~Pan, C.~Zhang, and K.~P. Lam, ``Whole building energy model for hvac optimal control: A practical framework based on deep reinforcement learning,'' \emph{Energy and Buildings}, vol. 199, pp. 472--490, 2019.

\bibitem{Taboga2021}
V.~Taboga, A.~Bellahsen, and H.~Dagdougui, ``An enhanced adaptivity of reinforcement learning-based temperature control in buildings using generalized training,'' \emph{IEEE Transactions on Emerging Topics in Computational Intelligence}, vol.~6, no.~2, pp. 255--266, 2022.

\bibitem{Du2021}
\BIBentryALTinterwordspacing
Y.~Du, H.~Zandi, O.~Kotevska, K.~Kurte, J.~Munk, K.~Amasyali, E.~Mckee, and F.~Li, ``Intelligent multi-zone residential hvac control strategy based on deep reinforcement learning,'' \emph{Applied Energy}, vol. 281, p. 116117, 2021. [Online]. Available: \url{https://www.sciencedirect.com/science/article/pii/S030626192031535X}
\BIBentrySTDinterwordspacing

\bibitem{Zhuang2023}
\BIBentryALTinterwordspacing
D.~Zhuang, V.~J. Gan, Z.~{Duygu Tekler}, A.~Chong, S.~Tian, and X.~Shi, ``Data-driven predictive control for smart hvac system in iot-integrated buildings with time-series forecasting and reinforcement learning,'' \emph{Applied Energy}, vol. 338, p. 120936, 2023. [Online]. Available: \url{https://www.sciencedirect.com/science/article/pii/S0306261923003008}
\BIBentrySTDinterwordspacing

\bibitem{yao2021}
\BIBentryALTinterwordspacing
Y.~Yao and D.~K. Shekhar, ``State of the art review on model predictive control (mpc) in heating ventilation and air-conditioning (hvac) field,'' \emph{Building and Environment}, vol. 200, p. 107952, 2021. [Online]. Available: \url{https://www.sciencedirect.com/science/article/pii/S0360132321003565}
\BIBentrySTDinterwordspacing

\bibitem{Yu2021}
L.~Yu, Y.~Sun, Z.~Xu, C.~Shen, D.~Yue, T.~Jiang, and X.~Guan, ``Multi-agent deep reinforcement learning for hvac control in commercial buildings,'' \emph{IEEE Transactions on Smart Grid}, vol.~12, no.~1, pp. 407--419, 2021.

\bibitem{Drgona2020}
J.~Drgoňa, J.~Arroyo, I.~{Cupeiro Figueroa}, D.~Blum, K.~Arendt, D.~Kim, E.~P. Ollé, J.~Oravec, M.~Wetter, D.~L. Vrabie, and L.~Helsen, ``All you need to know about model predictive control for buildings,'' \emph{Annual Reviews in Control}, vol.~50, pp. 190--232, 2020.

\bibitem{walker2017}
\BIBentryALTinterwordspacing
S.~S. Walker, W.~Lombardi, S.~Lesecq, and S.~Roshany-Yamchi, ``Application of distributed model predictive approaches to temperature and co2 concentration control in buildings,'' \emph{IFAC-PapersOnLine}, vol.~50, no.~1, pp. 2589--2594, 2017, 20th IFAC World Congress. [Online]. Available: \url{https://www.sciencedirect.com/science/article/pii/S2405896317301416}
\BIBentrySTDinterwordspacing

\bibitem{Hou2017}
X.~Hou, Y.~Xiao, J.~Cai, J.~Hu, and J.~E. Braun, ``{Distributed model predictive control via Proximal Jacobian ADMM for building control applications},'' \emph{Proceedings of the American Control Conference}, pp. 37--43, 6 2017.

\bibitem{Wang2022}
\BIBentryALTinterwordspacing
Z.~Wang, Y.~Zhao, C.~Zhang, P.~Ma, and X.~Liu, ``A general multi agent-based distributed framework for optimal control of building hvac systems,'' \emph{Journal of Building Engineering}, vol.~52, p. 104498, 2022. [Online]. Available: \url{https://www.sciencedirect.com/science/article/pii/S2352710222005113}
\BIBentrySTDinterwordspacing

\bibitem{Mork2022}
\BIBentryALTinterwordspacing
M.~Mork, A.~Xhonneux, and D.~Müller, ``Nonlinear distributed model predictive control for multi-zone building energy systems,'' \emph{Energy and Buildings}, vol. 264, p. 112066, 2022. [Online]. Available: \url{https://www.sciencedirect.com/science/article/pii/S0378778822002377}
\BIBentrySTDinterwordspacing

\bibitem{Drgona2021}
\BIBentryALTinterwordspacing
J.~Drgoňa, A.~R. Tuor, V.~Chandan, and D.~L. Vrabie, ``Physics-constrained deep learning of multi-zone building thermal dynamics,'' \emph{Energy and Buildings}, vol. 243, p. 110992, 2021. [Online]. Available: \url{https://www.sciencedirect.com/science/article/pii/S0378778821002760}
\BIBentrySTDinterwordspacing

\bibitem{Boyd2011}
S.~Boyd, N.~Parikh, E.~Chu, B.~Peleato, and J.~Eckstein, ``Distributed optimization and statistical learning via the alternating direction method of multipliers,'' \emph{Found. Trends Mach. Learn.}, vol.~3, no.~1, p. 1–122, jan 2011.

\bibitem{Yang2021}
Q.~Yang and H.~Wang, ``{Distributed energy trading management for renewable prosumers with HVAC and energy storage},'' \emph{Energy Reports}, vol.~7, pp. 2512--2525, 11 2021.

\bibitem{Rezaei2022}
E.~Rezaei, H.~Dagdougui, and M.~Rezaei, ``Distributed stochastic model predictive control for peak load limiting in networked microgrids with building thermal dynamics,'' \emph{IEEE Transactions on Smart Grid}, vol.~13, no.~3, pp. 2038--2049, 2022.

\bibitem{Ma2011}
Y.~Ma, G.~Anderson, and F.~Borrelli, ``{A distributed predictive control approach to building temperature regulation},'' \emph{Proceedings of the American Control Conference}, pp. 2089--2094, 2011.

\bibitem{Zhang2017}
\BIBentryALTinterwordspacing
X.~Zhang, W.~Shi, B.~Yan, A.~Malkawi, and N.~Li, ``{Decentralized and Distributed Temperature Control via HVAC Systems in Energy Efficient Buildings},'' 2 2017. [Online]. Available: \url{https://arxiv.org/abs/1702.03308v1}
\BIBentrySTDinterwordspacing

\bibitem{Hong2014}
M.~Hong, Z.-Q. Luo, and M.~Razaviyayn, ``Convergence analysis of alternating direction method of multipliers for a family of nonconvex problems,'' \emph{SIAM Journal on Optimization}, vol.~26, no.~1, pp. 337--364, 2016.

\bibitem{Hong2018AApproach}
M.~Hong, ``{A distributed, asynchronous, and incremental algorithm for nonconvex optimization: An ADMM approach},'' \emph{IEEE Transactions on Control of Network Systems}, vol.~5, no.~3, pp. 935--945, 9 2018.

\bibitem{Lu2022}
\BIBentryALTinterwordspacing
C.~Lu, S.~Li, and Z.~Lu, ``Building energy prediction using artificial neural networks: A literature survey,'' \emph{Energy and Buildings}, vol. 262, p. 111718, 2022. [Online]. Available: \url{https://www.sciencedirect.com/science/article/pii/S0378778821010021}
\BIBentrySTDinterwordspacing

\bibitem{Hafner2019}
D.~Hafner, T.~Lillicrap, I.~Fischer, R.~Villegas, D.~Ha, H.~Lee, and J.~Davidson, ``Learning latent dynamics for planning from pixels,'' in \emph{Proceedings of the 36th International Conference on Machine Learning}, ser. Proceedings of Machine Learning Research, K.~Chaudhuri and R.~Salakhutdinov, Eds., vol.~97.\hskip 1em plus 0.5em minus 0.4em\relax PMLR, 09--15 Jun 2019, pp. 2555--2565.

\bibitem{Hafner2020}
D.~Hafner, T.~Lillicrap, J.~Ba, and M.~Norouzi, ``Dream to control: Learning behaviors by latent imagination,'' 2020.

\bibitem{IEA2022}
\BIBentryALTinterwordspacing
IEA. (2022) Residential behaviour changes lead to a reduction in heating and cooling energy use by 2030. Paris. License: CC BY 4.0. [Online]. Available: \url{https://www.iea.org/reports/residential-behaviour-changes-lead-to-a-reduction-in-heating-and-cooling-energy-use-by-2030}
\BIBentrySTDinterwordspacing

\bibitem{Sunardi2020}
\BIBentryALTinterwordspacing
C.~Sunardi, Y.~P. Hikmat, A.~S. Margana, K.~Sumeru, and M.~F.~B. Sukri, ``Effect of room temperature set points on energy consumption in a residential air conditioning,'' \emph{AIP Conference Proceedings}, vol. 2248, no.~1, p. 070001, 2020. [Online]. Available: \url{https://aip.scitation.org/doi/abs/10.1063/5.0018806}
\BIBentrySTDinterwordspacing

\bibitem{Lamoudi2011}
M.~Y. Lamoudi, M.~Alamir, and P.~B{\'{e}}guery, ``{Distributed constrained Model Predictive Control based on bundle method for building energy management},'' \emph{Proceedings of the IEEE Conference on Decision and Control}, pp. 8118--8124, 2011.

\bibitem{DOEreferencebuildings}
\BIBentryALTinterwordspacing
``Prototype building models,'' 2021, accessed in April 2024. [Online]. Available: \url{https://www.energycodes.gov/prototype-building-models}
\BIBentrySTDinterwordspacing

\bibitem{Wang2020}
\BIBentryALTinterwordspacing
Z.~Wang, T.~Hong, and M.~A. Piette, ``Building thermal load prediction through shallow machine learning and deep learning,'' \emph{Applied Energy}, vol. 263, p. 114683, 2020. [Online]. Available: \url{https://www.sciencedirect.com/science/article/pii/S0306261920301951}
\BIBentrySTDinterwordspacing

\bibitem{Miller2018}
\BIBentryALTinterwordspacing
S.~D. Miller, M.~A. Rogers, J.~M. Haynes, M.~Sengupta, and A.~K. Heidinger, ``Short-term solar irradiance forecasting via satellite/model coupling,'' \emph{Solar Energy}, vol. 168, pp. 102--117, 2018, advances in Solar Resource Assessment and Forecasting. [Online]. Available: \url{https://www.sciencedirect.com/science/article/pii/S0038092X17310435}
\BIBentrySTDinterwordspacing

\bibitem{HydroQuebecDR}
\BIBentryALTinterwordspacing
``Hydro-quebec peak load management,'' 2023, accessed in May 2024. [Online]. Available: \url{https://www.hydroquebec.com/affaires/espace-clients/tarifs/option-gestion-demande-puissance.html}
\BIBentrySTDinterwordspacing

\bibitem{Djeunou2022}
F.~Djeumou, C.~Neary, E.~Goubault, S.~Putot, and U.~Topcu, ``Neural networks with physics-informed architectures and constraints for dynamical systems modeling,'' in \emph{Proceedings of The 4th Annual Learning for Dynamics and Control Conference}, ser. Proceedings of Machine Learning Research, R.~Firoozi, N.~Mehr, E.~Yel, R.~Antonova, J.~Bohg, M.~Schwager, and M.~Kochenderfer, Eds., vol. 168.\hskip 1em plus 0.5em minus 0.4em\relax PMLR, 23--24 Jun 2022, pp. 263--277.

\bibitem{BERTSEKAS19821}
D.~P. Bertsekas, ``Chapter 1 - introduction,'' in \emph{Constrained Optimization and Lagrange Multiplier Methods}.\hskip 1em plus 0.5em minus 0.4em\relax Academic Press, 1982, pp. 1--94.

\end{thebibliography}


\appendix
\section{Appendix}

\subsection{Proof of Theorem 1}
\label{subsec:proof_theorem_1}

For convenience, let us re-state here the problem. Note that to ease the notation we rename the decision variables of the problem by $\boldsymbol{x}$ instead of $\boldsymbol{\Delta}$. Recall that for $i = 1, \ldots, N$, we have $\boldsymbol{x}_i \in \mathbb{R}^H$, $\boldsymbol{\Bar{x}}_i \in \mathbb{R}^H$, $g(\boldsymbol{x}_i) = \|\boldsymbol{x}_i\|^2_2$ and $B_i \in \mathbb{R}^{NH \times H}$ such that $\boldsymbol{x} = \displaystyle \sum_{i=1}^N B_i \boldsymbol{x}_i$, where $\boldsymbol{x}\in \mathbb{R}^{NH}$. In the following, to ease the notation we drop the subscript to write $\{\boldsymbol{x}_i\}_{i=1,\ldots, N} = \{\boldsymbol{x}_i\}$. 

Consider the following problem :

\begin{align}
    &\underset{\boldsymbol{x}_1, \ldots, \boldsymbol{x}_N}{\text{minimize}}  \sum_{i=1}^N g(\boldsymbol{x}_i) +  \ell\left(\sum_{i=1}^N B_i \boldsymbol{\Bar{x}}_i \right) \label{eq:opt_problem_appendix} \\
    &\text{subject to: } & \\
    &\boldsymbol{\Bar{x}}_i = \boldsymbol{x}_i, \quad i=1, \ldots, N \\
    &\boldsymbol{x}_i \in \chi_i,  \quad i=1, \ldots, N
\end{align}
 
Where $\chi_i$ is some closed convex set. The problem is solved using Algorithm \ref{alg:non_convex_admm}. The augmented Lagrangian of this problem is 

\begin{align}
    \mathcal{L}(\{\boldsymbol{x}_i\}, \{\Bar{\boldsymbol{x}_i}\}, \{\boldsymbol{\lambda}_i\})  = \sum_{i=1}^N g(\boldsymbol{x}_i) + \ell \left(\sum_{i=1}^N B_i \Bar{\boldsymbol{x}}_i \right) \\
     + \sum_{i=1}^N \boldsymbol{\lambda}_i^T (\Bar{\boldsymbol{x}}_i-\boldsymbol{x}_i)  + \sum_{i=1}^N \frac{\rho}{2} \|\Bar{\boldsymbol{x}}_i-\boldsymbol{x}_i\|^2_2. \nonumber
\end{align}

The proof follows the same steps as in \cite{Hong2014}, but some differences arise because we are considering $N$ duplicated variables. Yet, as we show, the theorem holds mainly because of the orthogonality of the duplicated variables in $\mathbb{R}^{NH}$. In addition, note that compared to \cite{Hong2014} we consider an update of every variable at each iteration.

First, we state the following important properties :

\begin{proposition}
The function $g$ is convex.
\end{proposition}

\begin{proposition}
For $i =1, \ldots, N$ and  $k \in \mathbb{N}$,

$ \displaystyle \| B_i \boldsymbol{x}_i^{k+1} - B_i \boldsymbol{x}_i^k \|_2 = \| \boldsymbol{x}_i^{k+1} - \boldsymbol{x}_i^k \|_2$. 
\end{proposition}
\begin{proof}
    The vector $B_i \boldsymbol{x}_i^{k}$ has zero entries except in the positions $iH$ to $(i+N)H$ where it has the entries of $\boldsymbol{x}_i^k$.
\end{proof}

\begin{proposition} 
  \label{prop:orthogonal_xi}
  For $k \in \mathbb{N}$, 
  
  $\displaystyle \| \Bar{\boldsymbol{x}}^{k+1} - \Bar{\boldsymbol{x}}^{k} \|_2^2 = \sum_{i=1}^N \| \Bar{\boldsymbol{x}}_i^{k+1} - \Bar{\boldsymbol{x}}_i^{k} \|_2^2$
\end{proposition}
\begin{proof}
    The vectors $\{B_i \boldsymbol{x}_i\}_{i=1,\ldots, N}$ are orthogonal, which yields 
    \begin{align}
        \| \Bar{\boldsymbol{x}}^{k+1} - \Bar{\boldsymbol{x}}^{k} \|_2^2 & = \left \| \sum_{i=1}^N B_i (\Bar{\boldsymbol{x}}_i^{k+1} - \Bar{\boldsymbol{x}}_i^{k} ) \right \|_2^2 \\
        & = \sum_{i=1}^N \left \| B_i (\Bar{\boldsymbol{x}}_i^{k+1} - \Bar{\boldsymbol{x}}_i^{k} ) \right \|_2^2 \\
        & = \sum_{i=1}^N \left \| \Bar{\boldsymbol{x}}_i^{k+1} - \Bar{\boldsymbol{x}}_i^{k} \right \|_2^2 
    \end{align}
\end{proof}

In addition, we make the following assumptions :

\begin{assumption}
\label{assum:convex_pb}
The penalty parameter $\rho$ is chosen large enough such that the problems of Algorithm \ref{alg:non_convex_admm} Line 5 and 7 are strongly convex with parameters $\{\gamma_i\}$ and $\Bar{\gamma}$ respectively. 
\end{assumption}
See \cite{BERTSEKAS19821} for more details about Assumption \ref{assum:convex_pb}.

\begin{assumption}
\label{assum:lips_grad}

\begin{enumerate}
    \item We suppose that $\ell$ has a Lipschitz continuous gradient, i.e. there exists a $L > 0$ such that for all $\textbf{x}$ and $\textbf{y}$, $ \displaystyle \|\nabla \ell(\textbf{x}) -  \nabla \ell(\textbf{y}) \|_2 \leq L \|\textbf{x}-\textbf{y} \|_2. $
    
    \item We suppose $\rho \geq L$ and $\rho \bar{\gamma} \geq 2 L^2$, where $\bar{\gamma}$ is the strong convexity parameter defined in Assumption \ref{assum:convex_pb}.
\end{enumerate}
\end{assumption}

\begin{assumption}
\label{assum:f_bounded}
The function $ \displaystyle f(\boldsymbol{x}_1, \ldots, \boldsymbol{x}_N) = \sum_{i=1}^N g(\boldsymbol{x}_i) + \ell\left(\sum_{i=1}^N B_i \boldsymbol{x}_i \right) $ is lower bounded over $\displaystyle \prod_{i=1}^N \chi_i$.
\end{assumption}

With this framework, we demonstrate the following lemmas: 

\begin{lemma}
\label{lemma:1}
For $i=1, \ldots, N$ 

\begin{equation}
    B_i^T \nabla \ell \left( \Bar{\boldsymbol{x}} \right) = - \boldsymbol{\lambda}_i^{k+1} \label{lemma1:eq1}
\end{equation}

and 

\begin{equation}
     L^2 \|\ \Bar{\boldsymbol{x}}_i^{k+1} - \Bar{\boldsymbol{x}}_i^{k} \|^2_2  \geq \|\boldsymbol{\lambda}_i^{k+1} - \boldsymbol{\lambda}_i^{k} \|^2_2
\end{equation}
    
where $\nabla_{\Bar{\boldsymbol{x}}} \ell \in \mathbb{R}^{NH}$ is the gradient of $\ell$ with respect to $ \displaystyle \Bar{\boldsymbol{x}}= \sum_{i=1}^N B_i \Bar{\boldsymbol{x}}_i$.
\end{lemma}

\begin{proof}
The first order optimality condition of the problem Algorithm \ref{alg:non_convex_admm} Line 7 yields $N$ vector equalities, one for each $\Bar{\boldsymbol{x}}_i$ variable:
\begin{align}
    \nabla_{\Bar{\boldsymbol{x}}_i} \ell \left( \sum_{i=1}^N B_i \Bar{\boldsymbol{x}}_i \right) + \boldsymbol{\lambda}_i^k + \rho(\Bar{\boldsymbol{x}}_i - \boldsymbol{x}_i^{k+1}) = 0.
\end{align}  

Using the dual update (Algorithm \ref{alg:non_convex_admm} Line 9), the $\boldsymbol{x}_i^{k+1}$ are such that

\begin{align}
    &\nabla_{\Bar{\boldsymbol{x}}_i} \ell \left( \sum_{i=1}^N B_i \Bar{\boldsymbol{x}}_i^{k+1} \right)  = - \boldsymbol{\lambda}_i^{k+1} \\
    & B_i^T \nabla \ell \left( \Bar{\boldsymbol{x}}^{k+1} \right) = - \boldsymbol{\lambda}_i^{k+1}.
\end{align}

To prove the inequality, we use Assumption \ref{assum:lips_grad} and the equality (\ref{lemma1:eq1}) :

\begin{align}
    &\| B_i^T \nabla \ell \left( \Bar{\boldsymbol{x}}^{k+1} \right)  - B_i^T \nabla \ell \left( \Bar{\boldsymbol{x}}^k \right)  \|^2_2  = \|\boldsymbol{\lambda}_i^{k+1} - \boldsymbol{\lambda}_i^{k} \|^2_2 \\ 
    &L^2 \| \Bar{\boldsymbol{x}}_i^{k+1} - \Bar{\boldsymbol{x}}_i^k \|^2_2  \geq \|\boldsymbol{\lambda}_i^{k+1} - \boldsymbol{\lambda}_i^{k} \|^2_2 
\end{align}
\end{proof}

\begin{lemma}
\label{lemma:2}
The difference of two consecutive values of the augmented Lagrangian is bounded by the following negative quantity :
\begin{align}
&\mathcal{L}(\{\boldsymbol{x}_i^{k+1}\}, \{\Bar{\boldsymbol{x}}_i^{k+1}\}, \{\boldsymbol{\lambda}_i^{k+1}\}) - \mathcal{L}(\{\boldsymbol{x}_i^{k}\}, \{\Bar{\boldsymbol{x}}_i^{k}\}, \{\boldsymbol{\lambda}_i^{k}\}) \\ 
&\leq \sum_{i=1}^N  \left ( \frac{- \gamma_i}{2} \| \boldsymbol{x}^{k+1}_i - \boldsymbol{x}^k_i \|^2 - \left( \frac{\Bar{\gamma}}{2} - \frac{L^2}{\rho} \right) \|\Bar{\boldsymbol{x}}_i^{k+1} - \Bar{\boldsymbol{x}}_i^{k}\|^2 \right) \nonumber
\end{align}

\end{lemma}

\begin{proof}
The idea is to split the difference of two consecutive values of the augmented Lagrangian to explicit the update of each variable:

\begin{align}
    &\mathcal{L}(\{\boldsymbol{x}_i^{k+1}\}, \{\Bar{\boldsymbol{x}}_i^{k+1}\}, \{\boldsymbol{\lambda}_i^{k+1}\}) - \mathcal{L}(\{\boldsymbol{x}_i^{k}\}, \{\Bar{\boldsymbol{x}}_i^{k}\}, \{\boldsymbol{\lambda}_i^{k}\}) \\
    & = \mathcal{L}(\{\boldsymbol{x}_i^{k+1}\}, \{\Bar{\boldsymbol{x}}_i^{k+1}\}, \{\boldsymbol{\lambda}_i^{k+1}\}) - \mathcal{L}(\{\boldsymbol{x}_i^{k+1}\}, \{\Bar{\boldsymbol{x}}_i^{k+1}\}, \{\boldsymbol{\lambda}_i^{k}\}) \nonumber \\
    &  \quad \quad+ \mathcal{L}(\{\boldsymbol{x}_i^{k+1}\}, \{\Bar{\boldsymbol{x}}_i^{k+1}\}, \{\boldsymbol{\lambda}_i^{k}\}) - \mathcal{L}(\{\boldsymbol{x}_i^{k}\}, \{\Bar{\boldsymbol{x}}_i^{k+1}\}, \{\boldsymbol{\lambda}_i^{k}\})  \nonumber \\ 
    & \quad \quad  + \mathcal{L}(\{\boldsymbol{x}_i^{k}\}, \{\Bar{\boldsymbol{x}}_i^{k+1}\}, \{\boldsymbol{\lambda}_i^{k}\}) - \mathcal{L}(\{\boldsymbol{x}_i^{k}\}, \{\Bar{\boldsymbol{x}}_i^{k}\}, \{\boldsymbol{\lambda}_i^{k}\}).\nonumber
\end{align}

First, using the dual update rule (Algorithm \ref{alg:non_convex_admm} Line 9) and Lemma \ref{lemma:1} we have

\begin{align}
    &\mathcal{L}(\{\boldsymbol{x}_i^{k+1}\}, \{\Bar{\boldsymbol{x}}_i^{k+1}\}, \{\boldsymbol{\lambda}_i^{k+1}\}) - \mathcal{L}(\{\boldsymbol{x}_i^{k+1}\}, \{\Bar{\boldsymbol{x}}_i^{k+1}\}, \{\boldsymbol{\lambda}_i^{k}\}) \nonumber \\
    & \quad \quad = \frac{1}{\rho} \sum_{i=1}^N \| \boldsymbol{\lambda}_i^{k+1} - \boldsymbol{\lambda}_i^{k}\|^2_2 \\
    & \quad \quad \leq \frac{L^2}{\rho} \sum_{i=1}^N \|  \Bar{\boldsymbol{x}}_i^{k+1} - \Bar{\boldsymbol{x}}_i^{k} \|^2_2 \label{eq:lemma2_1}
\end{align}

In addition, $\mathcal{L}$ is strongly convex (Assumption \ref{assum:convex_pb}) so

\begin{align}
    & \mathcal{L}(\{\boldsymbol{x}_i^{k+1}\}, \{\Bar{\boldsymbol{x}}_i^{k+1}\}, \{\boldsymbol{\lambda}_i^{k}\}) - \mathcal{L}(\{\boldsymbol{x}_i^{k}\}, \{\Bar{\boldsymbol{x}}_i^{k+1}\}, \{\boldsymbol{\lambda}_i^{k}\}) \nonumber \\
    & \quad + \mathcal{L}(\{\boldsymbol{x}_i^{k}\}, \{\Bar{\boldsymbol{x}}_i^{k+1}\}, \{\boldsymbol{\lambda}_i^{k}\}) - \mathcal{L}(\{\boldsymbol{x}_i^{k}\}, \{\Bar{\boldsymbol{x}}_i^{k}\}, \{\boldsymbol{\lambda}_i^{k}\}) \nonumber \\ 
    & \leq \sum_{j=1}^N \Bigl( \bigl[ \nabla_{\boldsymbol{x}_j} \mathcal{L}(\{\boldsymbol{x}_i^{k+1}\}, \{\Bar{\boldsymbol{x}}_i^{k+1}\}, \{\boldsymbol{\lambda}_i^{k}\}) \bigr]^T (\boldsymbol{x}_j^{k+1} - \boldsymbol{x}_j^k) \nonumber \\
    & \hspace{4cm} - \frac{\gamma_j}{2} \| \boldsymbol{x}_j^{k+1} - \boldsymbol{x}_j^k \|^2_2 \Bigr) \nonumber \\
    & \hspace{2cm} + \Bigl[ \nabla_{\Bar{\boldsymbol{x}}} \mathcal{L}(\{\boldsymbol{x}_i^{k}\}, \Bar{\boldsymbol{x}}, \{\boldsymbol{\lambda}_i^{k}\}) \Bigr]^T (\Bar{\boldsymbol{x}}_i^{k+1} - \Bar{\boldsymbol{x}}_i^k) \nonumber \\
    & \hspace{2cm} - \frac{\Bar{\gamma}_i(\rho)}{2} \| \Bar{\boldsymbol{x}}^{k+1} - \Bar{\boldsymbol{x}}^k \|^2_2  \nonumber \\
    & \leq \sum_{i=1}^N \Bigl( - \frac{\gamma_i}{2} \| \boldsymbol{x}_i^{k+1} - \boldsymbol{\lambda}_i^k \|^2_2  - \frac{\Bar{\gamma}}{2} \| \Bar{\boldsymbol{x}}_i^{k+1} - \Bar{\boldsymbol{x}}_i^k \|_2^2 \Bigr) \label{eq:lemma2_2}
\end{align}

The last equality comes from Proposition \ref{prop:orthogonal_xi} and the optimality of the subproblems. Combining the two inequalities (\ref{eq:lemma2_1}) and (\ref{eq:lemma2_2}) leads to the result.
\end{proof}

\begin{lemma}
\label{lemma:3}
The limit $ \displaystyle \lim_{k \rightarrow \infty} \mathcal{L}(\{\boldsymbol{x}_i^k\}, \{\Bar{\boldsymbol{x}}_i^k\}, \{\boldsymbol{\lambda}_i^k\})$ exists and is bounded from below by $ \underset{\boldsymbol{x}_1, \ldots \boldsymbol{x}_N}{min} f(\boldsymbol{x}_1, \ldots, \boldsymbol{x}_n)$.
\end{lemma}

\begin{proof}
    First, we use the descent property of $\ell$ under the Assumption \ref{assum:lips_grad} :

\begin{align}
    \ell \left( \boldsymbol{x}^{k+1} \right)& \leq \ell (\Bar{\boldsymbol{x}}^{k+1}) + \Bigl[\nabla \ell (\Bar{\boldsymbol{x}}^{k+1}) \Bigr]^T \Bigl( \boldsymbol{x}^{k+1} - \Bar{\boldsymbol{x}}^{k+1} \Bigr) \nonumber \\
    & \quad \quad + \frac{L}{2} \Bigl\| \boldsymbol{x}^{k+1} - \Bar{\boldsymbol{x}}^{k+1} \Bigr\|^2_2 \nonumber \\
    & \leq \ell (\Bar{\boldsymbol{x}}^{k+1}) + \sum_{i=1}^N \Bigl( \Bigl[\nabla_{\boldsymbol{\Bar{x}}_i} \ell (\Bar{\boldsymbol{x}}^{k+1}) \Bigr]^T \Bigl( \boldsymbol{x}_i^{k+1} -\Bar{\boldsymbol{x}}_i^{k+1} \Bigr) \nonumber \\
    & \quad \quad + \frac{L}{2} \Bigl\| \boldsymbol{x}_i^{k+1} - \Bar{\boldsymbol{x}}_i^{k+1} \Bigr\|^2_2 \Bigr). \nonumber
\end{align}

Plugging this result in the expression of the Lagrangian yields:

\begin{align}
    & \mathcal{L}(\{\boldsymbol{x}_i^{k+1}\}, \{\Bar{\boldsymbol{x}}_i^{k+1}\}, \{\boldsymbol{\lambda}_i^{k+1}\}) \nonumber\\
    & \quad  = \sum_{i=1}^N g(\boldsymbol{x}_i^{k+1}) + \ell \left( \Bar{\boldsymbol{x}}^{k+1} \right) + \sum_{i=1}^N \boldsymbol{\lambda}_i^{k+1,T} (\Bar{\boldsymbol{x}}_i^{k+1}-\boldsymbol{x}_i^{k+1}) \\
    & \hspace{2cm} + \sum_{i=1}^N \frac{\rho}{2} \|\Bar{\boldsymbol{x}}_i^{k+1}-\boldsymbol{x}_i^{k+1}\|_2^2 \nonumber \\
    & \quad  = \sum_{i=1}^N g(\boldsymbol{x}_i^{k+1}) + \ell \left( \Bar{\boldsymbol{x}}^{k+1} \right) \nonumber \\
    & \hspace{2cm} + \sum_{i=1}^N \left[ \nabla_{\boldsymbol{\Bar{x}}_i} \ell (\Bar{\boldsymbol{x}}^{k+1}) \right]^T \left(\boldsymbol{x}_i^{k+1} - \Bar{\boldsymbol{x}}_i^{k+1} \right) \nonumber \\
    &\hspace{2cm} + \sum_{i=1}^N \frac{\rho}{2} \|\Bar{\boldsymbol{x}}_i^{k+1}-\boldsymbol{x}_i^{k+1}\|_2^2 \label{eq:lemma3_ineq}\\
    & \quad \geq \sum_{i=1}^N g(\boldsymbol{x}_i^{k+1}) + \ell \left(\boldsymbol{x}^{k+1} \right) + \sum_{i=1}^N \frac{\rho-L}{2} \|\Bar{\boldsymbol{x}}_i^{k+1}-\boldsymbol{x}_i^{k+1}\|_2^2 \\
    &\quad = f(\boldsymbol{x}^{k+1}) + \sum_{i=1}^N \frac{\rho-L}{2} \|\Bar{\boldsymbol{x}}_i^{k+1}-\boldsymbol{x}_i^{k+1}\|_2^2
\end{align}

where (\ref{eq:lemma3_ineq}) comes from Lemma \ref{lemma:1}.

Since $ \displaystyle f$ is bounded from below (Assumption \ref{assum:f_bounded}), so is $ \displaystyle \mathcal{L}(\{\boldsymbol{x}_i^{k+1}\}, \{\Bar{\boldsymbol{x}}_i^{k+1}\}, \{\boldsymbol{\lambda}_i^{k+1}\})$. 

In addition, Lemma \ref{lemma:2} shows that the sequence $\{ \mathcal{L}(\{\boldsymbol{x}_i^{k+1}\}, \{\Bar{\boldsymbol{x}}_i^{k+1}\}, \{\boldsymbol{\lambda}_i^{k+1}\} \}_k$ is decreasing, which leads to the result. 
\end{proof}

\begin{lemma}
\label{lemma:4}
For $i=1, \ldots, N$, $$ \lim_{k \xrightarrow{}\infty} \| \Bar{\boldsymbol{x}}_i^k  - \boldsymbol{x}_i^k\|_2 =0  $$
\end{lemma}

\begin{proof}

We have shown that the sequence $\{ \mathcal{L}(\{\boldsymbol{x}_i^{k+1}\}, \{\Bar{\boldsymbol{x}}_i^{k+1}\}, \{\boldsymbol{\lambda}_i^{k+1}\} \}_k$ converges. From Lemma \ref{lemma:2} we have

\begin{equation}
    \lim_{k \xrightarrow{} \infty} \| \boldsymbol{x}_i^{k+1} - \boldsymbol{x}_i^k \|_2 = 0 \label{eq:lim1}
\end{equation}

and 

\begin{equation}
    \lim_{k \xrightarrow{} \infty} \| \Bar{\boldsymbol{x}}_i^{k+1} - \Bar{\boldsymbol{x}}_i^k \|_2 = 0 \label{eq:lim2}.
\end{equation}

In addition, since $L^2 \| \Bar{\boldsymbol{x}}_i^{k+1} - \Bar{\boldsymbol{x}}_i^k \|^2_2 \geq \| \boldsymbol{\lambda}_i^{k+1} - \boldsymbol{\lambda}_i^k \|_2^2$ for all $i$, we have :

\begin{equation}
    \lim_{k \xrightarrow{} \infty} \| \boldsymbol{\lambda}_i^{k+1} - \boldsymbol{\lambda}_i^k \|_2 = 0 \label{eq:lim3}.
\end{equation}

The dual update rule (Algorithm \ref{alg:non_convex_admm} Line 9) then ensures that $$ \lim_{k \xrightarrow{}\infty} \| \Bar{\boldsymbol{x}}_i^k  - \boldsymbol{x}_i^k\|_2 =0.  $$

\end{proof}

\begin{lemma}
\label{lemma:5}

Any limit point $(\{\boldsymbol{x}_i^*\}, \{\bar{\boldsymbol{x}}_i^*\}, \{\boldsymbol{\lambda}_i^*\})$ of the sequence $(\{\boldsymbol{x}_i\}, \{\bar{\boldsymbol{x}}_i\}, \{ \boldsymbol{\lambda}_i\})$ generated by Algorithm \ref{alg:non_convex_admm} is a stationary solution of the problem (\ref{eq:opt_problem_appendix}):

\begin{align*}
    & (\nabla g(\boldsymbol{x}_i^*) - \boldsymbol{\lambda}_i^*)^T(\boldsymbol{y}-\boldsymbol{x}_i^*) \geq 0 \quad \forall \boldsymbol{y}_i \in \chi_i \\
    & \nabla_{\boldsymbol{\bar{x}}_i} \ell(\bar{\boldsymbol{x}}^*) - \boldsymbol{\lambda}_i^* = 0, \quad i=1, \ldots, N \\
    & \boldsymbol{x}_i^* = \Bar{\boldsymbol{x}}_i^* \quad i=1, \ldots, N 
\end{align*}

\end{lemma}

\begin{proof}

By taking the limit of the result of Lemma \ref{lemma:4}, we have that $\boldsymbol{x}_i^* = \Bar{\boldsymbol{x}}_i^*$ for all $i$. 

In addition, for all $y \in \chi_i$ the optimality condition of the problem in Algorithm \ref{alg:non_convex_admm} Line 5 is

\begin{align}
    & \left[ \nabla_{\boldsymbol{x}_i} g(\boldsymbol{x}_i^{k+1}) - \boldsymbol{\lambda}_i^k + \rho(\Bar{\boldsymbol{x}}_i^{k} -\boldsymbol{x}_i^{k+1}) \right]^T (\boldsymbol{y}-\boldsymbol{x}_i^{k+1}) \geq 0 \label{inequality_opt_condition}
\end{align}

We previously showed that for all $i$, $\displaystyle \lim_{k \xrightarrow{}\infty} \| \Bar{\boldsymbol{x}}_i^k  - \boldsymbol{x}_i^k\| = 0$. In addition, Eq. (\ref{eq:lim1},\ref{eq:lim2},\ref{eq:lim3}) give $ \displaystyle \lim_{k \xrightarrow{}\infty} \Bar{\boldsymbol{x}}_i^{k} = \Bar{\boldsymbol{x}}_i^*$, $\displaystyle \lim_{k \xrightarrow{}\infty} \boldsymbol{x}_i^k = \boldsymbol{x}_i^*$, $\displaystyle \lim_{k \xrightarrow{}\infty} \boldsymbol{\lambda}_i^k = \boldsymbol{\lambda}_i^*$. Taking the limit in the above inequality (\ref{inequality_opt_condition}) yields:

\begin{align}
        & (\nabla g(\boldsymbol{x}_i^*) - \boldsymbol{\lambda}_i^*)^T(\boldsymbol{y}-\boldsymbol{x}_i^*) \geq 0 \quad \forall \boldsymbol{y}_i \in \chi_i.
\end{align}

Taking the limit in the result of Lemma \ref{lemma:2} gives

\begin{align}
    \nabla_{\boldsymbol{\bar{x}}_i} \ell(\bar{\boldsymbol{x}}^*) + \boldsymbol{\lambda}_i^* = 0 .
\end{align}

\end{proof}

From these lemmas, the proof of Theorem \ref{theorem:1} is identical to the argument given in \cite{Hong2014} (Theroem 3.4), and is based on the compactness of the sets $\chi_i$. For completeness, we remind the theorem and state the argument here.

\begin{theorem}
    The sequence generated by Algorithm \ref{alg:non_convex_admm} converges to the set of stationary solutions of Problem (\ref{opt:admm_problem}) 

$$ \lim_{k\xrightarrow{}\infty} dist\left(\{\boldsymbol{x}_i^k\}, \{\Bar{\boldsymbol{x}}_i^k\}, \{\boldsymbol{\lambda}_i^k\}; Z^*\right) = 0,$$

where $Z^*$ is the set of primal-dual stationary solutions of the problem. The distance between a vector $\boldsymbol{x}$ and the set $Z^*$ is defined as:

$$ \text{dist} \left( \boldsymbol{x}; Z^*  \right) = \underset{y \in Z^*}{\text{min}} \| x-y \|_2$$

\end{theorem}

\begin{proof}    First, the compactness of the sets $\chi_i$ ensures for all $i$ the sequences $\{\boldsymbol{x}_i^k\}_k$ and $\{\bar{\boldsymbol{x}}_i^k\}_k$ have a limit point. In addition, the function $\ell$ has Lipschitz continuous gradients $\nabla_{\bar{\boldsymbol{x}}_i} \ell_i (\bar{\boldsymbol{x}}_i^k)$, so for all $i$ the set $\{\nabla_{\bar{\boldsymbol{x}}_i} \ell_i (\bar{\boldsymbol{x}}_i) \mid \bar{\boldsymbol{x}}_i \in \chi_i \}$ is bounded and $\{\nabla_{\bar{\boldsymbol{x}}_i} \ell_i (\bar{\boldsymbol{x}}_i^k)\}_k$ is a bounded sequence. From Eq. (\ref{lemma1:eq1}), $\displaystyle \nabla_{\bar{\boldsymbol{x}}_i} \ell_i (\bar{\boldsymbol{x}}_i^k) = - \boldsymbol{\lambda}^k_i$, so the sequences $\{\boldsymbol{\lambda}_i^k\}_k$ also have a limit point.

    The sequences $\{\boldsymbol{x}_i^k\}_k$, $\{\bar{\boldsymbol{x}}_i^k\}_k$, and $\{\boldsymbol{\lambda}_i^k\}_k$ lie in compact sets, so for all $i$ every subsequence $(\{\boldsymbol{x}_i^{k_j}\}, \{\bar{\boldsymbol{x}}_i^{k_j}\}, \{\boldsymbol{\lambda}_i^{k_j}\})$ has a limit point $(\{\hat{\boldsymbol{x}}_i\}, \{\hat{\bar{\boldsymbol{x}}}_i\}, \{\hat{\boldsymbol{\lambda}}_i\})$. From Lemma \ref{lemma:5} we know this limit point is in $Z^*$. By further restricting the subsequence if necessary, we can assume this limit point is unique.
    
    Suppose the sequence $(\{\boldsymbol{x}_i^{k_j}\}, \{\bar{\boldsymbol{x}}_i^{k_j}\}, \{\boldsymbol{\lambda}_i^{k_j}\})$ does not converge to $Z^*$. Then
    
    \begin{align}
        \lim_{j \rightarrow \infty} \text{dist}\left((\{\boldsymbol{x}_i^{k_j}\}, \{\bar{\boldsymbol{x}}_i^{k_j}\}, \{\boldsymbol{\lambda}_i^{k_j}\}); Z^*\right) = e > 0. \label{theorem:dist_eq}
    \end{align}
    
    But there exists an $n > 0$ such that for all $j \geq n$,
    
    \begin{align}
        \left\| (\{\boldsymbol{x}_i^{k_j}\}, \{\bar{\boldsymbol{x}}_i^{k_j}\}, \{\boldsymbol{\lambda}_i^{k_j}\}) - (\{\hat{\boldsymbol{x}}_i\}, \{\hat{\bar{\boldsymbol{x}}}_i\}, \{\hat{\boldsymbol{\lambda}}_i\}) \right\|_2^2 \leq \frac{e}{2},
    \end{align}
    
    and by definition of the distance,
    
    \begin{align}
        \text{dist}\left( (\{\boldsymbol{x}_i^{k_j}\}, \{\bar{\boldsymbol{x}}_i^{k_j}\}, \{\boldsymbol{\lambda}_i^{k_j}\}); Z^* \right) \leq \frac{e}{2},
    \end{align}
    
    which contradicts Eq. (\ref{theorem:dist_eq}). We have shown that every sequence of iterates has a limit point that is a stationary solution. Furthermore, the sequence of iterates generated by the algorithm converges to one of these limit points. The sequence of iterates produced by Algorithm \ref{alg:non_convex_admm} therefore converges to the set of stationary solutions.

\end{proof}

\vspace{2cm}

\textit{©2024 IEEE.  Personal use of this material is permitted.  Permission from IEEE must be obtained for all other uses, in any current or future media, including reprinting/republishing this material for advertising or promotional purposes, creating new collective works, for resale or redistribution to servers or lists, or reuse of any copyrighted component of this work in other works.}

\end{document}